\documentclass[12pt]{article}

\usepackage{graphicx}
\usepackage{amsmath}
\usepackage{amssymb}
\usepackage{latexsym}
\usepackage{subfigure}
\usepackage{crop}
\usepackage{algorithmic, algorithm}
\usepackage{multirow}%\usepackage{algorithm,algorithmic}
\usepackage[section,subsection,subsubsection]{placeins}
\usepackage[colorlinks = true, pdfstartview = FitV, linkcolor = blue, citecolor = blue, urlcolor = blue]{hyperref}
\usepackage{bm}
\usepackage{bbm}
\usepackage{enumerate}
\usepackage{framed} % or, "mdframed"
\usepackage[framed]{ntheorem}
\usepackage{array}
\usepackage{paralist}

\newcommand {\bgg}  { {\bf g} }

\newcommand {\xx}  { {\bf x} }
\renewcommand {\aa}  { {\bf a} }

\newcommand {\yy}  { {\bf y} }

\newcommand {\Ex} { {\mathbb E} }

\newcommand {\pp}  { {\bf p} }

\newcommand {\vv}  { {\bf v} }
\newcommand {\ww}  { {\bf w} }

\newcommand{\hf}{\frac12}

\newcommand{\defeq}{\mathrel{\mathop:}=}

\newcounter{comment}
\def\comment{\refstepcounter{comment}\textbf{Comment \arabic{comment}: }}

\theoremclass{Theorem}
\theoremstyle{break}
\newframedtheorem{theorem}{Theorem}
\newframedtheorem{corollary}{Corollary}
\newframedtheorem{lemma}{Lemma}

\newenvironment{proof}[1][Proof]{\begin{trivlist}
\item[\hskip \labelsep {\bfseries #1}]}{\end{trivlist}}

\newcommand{\qed}{\nobreak \ifvmode \relax \else
      \ifdim\lastskip<1.5em \hskip-\lastskip
      \hskip1.5em plus0em minus0.5em \fi \nobreak
      \vrule height0.75em width0.5em depth0.25em\fi}

\begin{document}

\pagestyle{plain} % No headers, just page numbers
\setcounter{page}{1}
%\title{Improved bounds on sample size for randomized Monte-Carlo trace estimators for implicit symmetric positive semi-definite matrices}
\title{Sub-Sampled Newton Methods I: Globally Convergent Algorithms}
\author{Farbod Roosta-Khorasani\thanks{International Computer Science Institute, Berkeley, CA 94704
 and Department of Statistics, University of California at Berkeley, Berkeley, CA 94720. 
{\tt farbod/mmahoney@stat.berkeley.edu}.} \and Michael W. Mahoney\footnotemark[1]}
%\thanks{{\tt farbod@stat.berkeley.ca}.}}
\maketitle
\begin{abstract}
Large scale optimization problems are ubiquitous in machine learning and data analysis and there is a plethora of algorithms for solving such problems. Many of these algorithms employ sub-sampling, as a way to either speed up the computations and/or to implicitly implement a form of statistical regularization. In this paper, we consider second-order iterative optimization algorithms, i.e., those that use Hessian as well as gradient information, and we provide bounds on the convergence of the variants of Newton's method that incorporate uniform sub-sampling as a means to estimate the gradient and/or Hessian.  Our bounds are non-asymptotic, i.e., they hold for finite number of data points in finite dimensions for finite number of iterations. In addition, they are quantitative and depend on the quantities related to the problem, i.e., the condition number. However, our algorithms are global and are guaranteed to converge from any initial iterate. 

Using random matrix concentration inequalities, one can sub-sample the Hessian in a way that the curvature information is preserved. Our first algorithm incorporates such sub-sampled Hessian while using the full gradient. We also give additional convergence results for when the sub-sampled Hessian is regularized by modifying its spectrum or ridge-type regularization. Next, in addition to Hessian sub-sampling, we also consider sub-sampling the gradient as a way to further reduce the computational complexity per iteration. We use approximate matrix multiplication results from randomized numerical linear algebra (RandNLA) to obtain the proper sampling strategy. In all these algorithms, computing the update boils down to solving a large scale linear system, which can be computationally expensive. As a remedy, for all of our algorithms, we also give global convergence results for the case of inexact updates where such linear system is solved only approximately. 

This paper has a more advanced companion paper~\cite{romassn2} in which we demonstrate that, by doing a finer-grained analysis, we can get problem-independent bounds for local convergence of these algorithms and explore tradeoffs to improve upon the basic results of the present paper.

%\center{\red{\textbf{For internal use only. Please do not distribute.}}}

\end{abstract}

\section{Introduction}
\label{sec:intro}
Large scale optimization problems arise frequently in machine learning and data analysis and there has been a great deal of effort to devise algorithms for efficiently solving such problems. Here, following many data-fitting applications, we consider the optimization problem of the form 
\begin{equation}
\min _{\xx \in  \mathbb{R}^{p}} F(\xx) = \frac{1}{n} \sum_{i=1}^{n} f_{i}(\xx),
\label{obj}
\end{equation}
where each $f_{i}: \mathbb{R}^{p} \rightarrow \mathbb{R}$ corresponds to an observation (or a measurement) which models the loss (or misfit) given a particular choice of the underlying parameter $\xx$. Examples of such optimization problems arise frequently in machine learning such as logistic regression, support vector machines, neural networks and graphical models. Many optimization algorithms have been developed to solve~\eqref{obj},~\cite{bertsekas1999nonlinear,nesterov2004introductory,boyd2004convex}. Here, we consider the high dimensional regime where both $p$ and $n$ are very large, i.e., $n, p \gg 1$. In such high dimensional settings, the mere evaluation of the gradient or the Hessian can be computationally prohibitive. As a result, many of the classical \textit{deterministic} optimization algorithms might prove to be inefficient, if applicable at all. In this light, there has been a great deal of effort to design \textit{stochastic} variants which are efficient and can solve the modern ``big data'' problems. Many of these algorithms employ sub-sampling as a way to speed up the computations. For example, a particularly simple version of~\eqref{obj} is when the $f_i$'s are quadratics, in which case one has very over-constrained least squares problem.  For these problems, randomized numerical linear algebra (RandNLA) has employed random sampling~\cite{mahoney2011randomized}, e.g., sampling  with respect to approximate leverage scores~\cite{drineas2006sampling,drineas2012fast}. Alternatively, on can perform a random projection followed by uniform sampling in the randomly rotated space~\cite{drineas2011faster}.  In these algorithms, sampling is used to get a data-aware or data-oblivious subspace embedding, i.e., an embedding which preserves the geometry of the entire subspace, and as such, one can get strong relative-error bounds of the solution.  Moreover, implementations of algorithms based on those ideas have been shown to beat state-of-the-art numerical routines~\cite{avron2010blendenpik,meng2014lsrn,yang2015implementing}.  For more general optimization problems of the form of~\eqref{obj}, optimization algorithms are common, and within the class of first order methods, i.e., those which only use gradient information, there are many corresponding results. However, within second order methods, i.e., the ones that use both the gradient and the Hessian information, one has yet to devise globally convergent algorithms with non-asymptotic convergence guarantees.  We do that here. In particular, we present sub-sampled ``Newton-type'' algorithms which are global and are guaranteed to converge from any initial iterate. Subsequently, we give convergence guarantees which are non-asymptotic, i.e., they hold for finite number of data points in finite dimensions for finite number of iterations. %In addition, our bounds are quantitative and depend on the quantities related to the problem, i.e., the condition number.  

The rest of this paper is organized as follows: in Section~\ref{sec:general_background}, we first give a very brief background on the general methodology for optimizing~\eqref{obj}. The notation and the assumptions used in this paper are given in Section~\ref{sec:not_assmpts}. The contributions of this paper are listed in Section~\ref{sec:contributions}. Section~\ref{sec:related_work} surveys the related work. Section~\ref{sec:subsampl_newton_hess} gives global convergence results for the case where only the Hessian is sub-sampled while the full gradient is used. In particular, Section~\ref{sec:global_con_nm_exact} gives a linearly convergent global algorithm with exact update, whereas Section~\ref{sec:global_con_nm_inexact} gives a similar result for the case when approximate solution of the linear system is used as search direction. The case where the gradient, as well as Hessian, is sub-sampled is treated in Section~\ref{sec:subsampl_newton_hess_grad}.  More specifically, Section~\ref{sec:global_con_nm_exact_grad} gives a globally convergent algorithm with linear rate with exact update, whereas Section~\ref{sec:global_con_nm_inexact_grad} addresses the algorithm with inexact updates. A few examples from generalized linear models (GLM), a very popular class of problems in machine learning community, as well as numerical simulations are given in Section~\ref{sec:examples}. Conclusions and further thoughts are gathered in Section~\ref{sec:conclusion}. All proofs are given in the appendix.

\subsection{General Background}
\label{sec:general_background}
For optimizing~\eqref{obj}, the standard deterministic or full gradient method, which dates back to Cauchy~\cite{cauchy1847methode}, uses iterations of the form 
\begin{equation*}
\xx^{(k+1)} = \xx^{(k)} - \alpha_{k} \nabla F(\xx^{(k)}),
\end{equation*}
where $\alpha_{k}$ is the step size at iteration $k$. However, when $n \gg 1$, the full gradient method can be inefficient because its iteration cost scales linearly in $n$. In addition, when $p \gg 1 $ or when each individual $f_{i}$ are complicated functions (e.g., evaluating each $f_{i}$ may require the solution of a partial differential equation), the mere evaluation of the gradient can be computationally prohibitive.  Consequently, stochastic variant of full gradient descent, e.g., (mini-batch) stochastic gradient descent (SGD) was developed~\cite{robbins1951stochastic,le2004large,li2014efficient,bertsekas1996neuro,bottou2010large,cotter2011better}. In such methods a subset $\mathcal{S} \subset \{1,2,\cdots,n\}$ is chosen at random and the update is obtained by
\begin{equation*}
\xx^{(k+1)} = \xx^{(k)} - \alpha_{k} \sum_{j \in \mathcal{S}} \nabla f_{j}(\xx^{(k)}).
\end{equation*}
When $|\mathcal{S}| \ll n$ (e.g., $|\mathcal{S}| = 1$ for simple SGD), the main advantage of such stochastic gradient methods is that the iteration cost is independent of $n$ and can be much cheaper than the full gradient methods, making them suitable for modern problems with large $n$. 

The above class of methods are among what is known as \textit{first-order} methods where only
the gradient information is used at every iteration. One attractive feature of such class of
methods is their relatively low per-iteration-cost.  Despite the low per-iteration-cost of first order methods, in almost all problems, incorporating curvature information (e.g., Hessian) as a form of scaling the gradient, i.e., 
\begin{equation*}
\xx^{(k+1)} = \xx^{(k)} - \alpha_{k} D_{k} \nabla F(\xx^{(k)}),
\end{equation*}
can significantly improve the convergence rate. Such class of methods which take the curvature information into account are known as \textit{second-order} methods, and compared to first-order methods, they enjoy superior convergence rate in both theory and practice. This is so since there is an implicit local \textit{scaling} of coordinates at a given $\xx$, which is determined by the local curvature of $F$. This local curvature in fact determines the condition number of a $F$ at $\xx$. Consequently, by taking the curvature information into account (e.g., in the form of the Hessian), second order methods can rescale the gradient direction so it is a much more ``useful'' direction to follow. This is in contrast to first order methods which can only scale the gradient uniformly for all coordinates. Such second order information have long been used in many machine learning applications~\cite{bottou1998online,yu2010quasi,lin2008trust,martens2010deep,byrd2011use,byrd2012sample}. 

The canonical example of second order methods, i.e., Newton's method~\cite{nesterov2004introductory,boyd2004convex,nocedal2006numerical}, is with $D_{k}$ taken to be the inverse of the full Hessian and $\alpha_{k}=1$, i.e.,
\begin{equation*}
\xx^{(k+1)} = \xx^{(k)} - [\nabla^{2} F(\xx^{(k)})]^{-1} \nabla F(\xx^{(k)}).
\end{equation*}
It is well known that for smooth and strongly convex function $F$, the Newton direction is always a descent direction and by introducing a step-size, $\alpha_{k}$, it is possible to guarantee the global convergence (by globally convergent algorithm, it is meant an algorithm that approaches the optimal solution starting from any initial point). %The resulting algorithm is known as \textit{damped} Newton's method. In fact, by adjusting the step-size of the damped Newton's method (e.g., using Armijo rule~\cite{armijo1966minimization}), one can indeed recover the asymptotic quadratic convergence rate. 
In addition, for cases where $F$ is not strongly convex, the Levenberg-Marquardt type regularization,~\cite{levenberg1944algorithm, marquardt1963algorithm}, of the Hessian can be used to obtain globally convergent algorithm. An important property of Newton's method is \textit{scale invariance}. More precisely, for some new parametrization $\tilde{\xx} = A \xx$ for some invertible matrix A, the optimal search direction in the new coordinate system is $\tilde{\pp} = A \pp$ where $\pp$ is the original optimal search direction. By contrast, the search direction produced by gradient descent behaves in an opposite fashion as $\tilde{\pp} = A^{-T} \pp$. Such scale invariance property is important to more effectively optimize poorly scaled parameters; see~\cite{martens2010deep} for a very nice and intuitive explanation of this phenomenon. 

However, when $n,p\gg1$, the per-iteration-cost of such algorithm is significantly higher than that of first-order methods. As a result, a line of research is to try to construct an approximation of the Hessian in a way that the update is computationally feasible, and yet, still provides sufficient second order information. One such class of methods are quasi-Newton methods, which are a generalization of the secant method to find the root of the first derivative for multidimensional problems. In such methods, the approximation to the Hessian is updated iteratively using only first order information from the gradients and the iterates through low-rank updates. Among these methods, the celebrated Broyden-Fletcher-Goldfarb-Shanno (BFGS) algorithm~\cite{nocedal2006numerical} and its limited memory version (L-BFGS)~\cite{nocedal1980updating,liu1989limited}, are the most popular and widely used members of this class. Another class of methods for approximating the Hessian is based on \textit{sub-sampling} where the Hessian of the full function $F$ is estimated using that of the randomly selected subset of functions $f_{i}$,~\cite{byrd2011use, byrd2012sample, erdogdu2015convergence,martens2010deep}. More precisely, a subset $\mathcal{S} \subset \{1,2,\cdots,n\}$ is chosen at random and, if the sub-sampled matrix is invertible, the update is obtained by 
\begin{subequations}
\begin{equation}
\xx^{(k+1)} = \xx^{(k)} - \alpha_{k} \big[\sum_{j \in \mathcal{S}} \nabla^{2} f_{j}(\xx^{(k)})\big]^{-1} \nabla F(\xx^{(k)}),
\label{sub_hess_unconstrained}
\end{equation}
In fact, sub-sampling can also be done for the gradient, obtaining a fully stochastic iteration
\begin{equation}
\xx^{(k+1)} = \xx^{(k)} - \alpha_{k} \big[\sum_{j \in \mathcal{S}_{H}} \nabla^{2} f_{j}(\xx^{(k)})\big]^{-1} \sum_{j \in \mathcal{S}_{\bgg}} \nabla f_{j}(\xx^{(k)}),
\label{sub_hess_grad_unconstrained}
\end{equation}
\label{sub_sampled_unconstrained}
\end{subequations}
where $\mathcal{S}_{\bgg}$ and $\mathcal{S}_{H}$ are sample sets used for approximating the Hessian and the gradient, respectively. The variants~\eqref{sub_sampled_unconstrained} are what we call \textit{sub-sampled Newton methods} in this paper.

This paper has a companion paper~\cite{romassn2}, henceforth called SSN2, which considers the technically-more-sophisticated local convergence rates for sub-sampled Newton methods (by local convergence, it is meant that the initial iterate is close enough to a local minimizer at which the sufficient conditions hold). However, here, we only concentrate on designing such algorithms with global convergence guarantees. In doing so, we need to ensure the following requirements:
\begin{enumerate}[(R.1)]
	\item \label{small_sample_size} Our sampling strategy needs to provide a sample size $|\mathcal{S}|$ which is independent of $n$, or at least smaller. Note that this is the same requirement as in SSN2~\cite[(R.1)]{romassn2}. However, as a result of the simpler goals of the present paper, we will show that, comparatively, a much smaller sample size can be required here than that used in SSN2~\cite{romassn2}.
	\item \label{invertibility} In addition, any such method must, at least probabilistically, ensure that the sub-sampled matrix is invertible. If the gradient is also sub-sampled, we need to ensure that sampling is done in a way to keep as much of this first order information as possible. Note that unlike SSN2~\cite[(R.2)]{romassn2} where we require a strong spectrum-preserving property, here, we only require the much weaker invertibility condition, which is enough to yield global convergence. 
		\item \label{global_rate} We need to ensure that our designed algorithms are globally convergent and approach the optimum starting from any initial guess. In addition, we require to have bounds which yield explicit convergence rate as opposed to asymptotic results. Note that unlike SSN2~\cite[(R.3)]{romassn2} where our focus is on speed, here, we mainly require global convergence guarantees.
	\item \label{inexact_solve} For $p \gg 1$, even when the sub-sampled Hessian is invertible, computing the update at each iteration can indeed pose a significant computational challenge. More precisely, it is clear from~\eqref{sub_sampled_unconstrained} that to compute the update, sub-sampled Newton methods require the solution of a linear system, which regardless of the sample size, can be the bottleneck of the computations. Solving such systems \textit{inexactly} can further improve the computational efficiency of sub-sampled algorithms. Hence, it is imperative to allow for approximate solutions and still guarantee convergence.
\end{enumerate}

In this paper, we give global convergence rates for sub-sampled Newton methods, addressing challenges~\hyperref[small_sample_size]{(R.1)},~\hyperref[invertibility]{(R.2)},~\hyperref[inexact_solve]{(R.3)} and~\hyperref[global_rate]{(R.4)}. As a result, the local rates of the companion paper, SSN2~\cite{romassn2}, coupled with the global convergence guarantees presented here, provide globally convergent  algorithms with \textit{fast} and \textit{problem-independent} local rates (e.g., see Theorems~\ref{local_global_newton} and~\ref{local_global_newton_grad}). To the best of our knowledge, the present paper and SSN2~\cite{romassn2} are the very first to thoroughly and quantitatively study the convergence behavior of such sub-sampled second order algorithms, in a variety of settings.

\subsection{Notation and Assumptions}
\label{sec:not_assmpts}
Throughout the paper, vectors are denoted by bold lowercase letters, e.g., $\vv$, and matrices or random variables are denoted by regular upper case letters, e.g., $V$, which is clear from the context. For a vector $\vv$, and a matrix $V$, $\|\vv\|$ and $\|V\|$ denote the vector $\ell_{2}$ norm and the matrix spectral norm, respectively, while $\|V\|_{F}$ is the matrix Frobenius norm. $\nabla f(\xx)$ and $\nabla^{2} f(\xx)$ are the gradient and the Hessian of $f$ at $\xx$, respectively and $\mathbb{I}$ denotes the identity matrix. For two symmetric matrices $A$ and $B$, $A \succeq B$ indicates that $A-B$ is symmetric positive semi-definite. The superscript, e.g., $\xx^{(k)}$, denotes iteration counter and $\ln(x)$ is the natural logarithm of $x$. Throughout the paper, $\mathcal{S}$ denotes a collection of indices from $\{1,2,\cdots,n\}$, with potentially repeated items and its cardinality is denoted by $|\mathcal{S}|$.

For our analysis throughout the paper, we make the following blanket assumptions: we require that each $f_{i}$ is twice-differentiable, smooth and convex, i.e., for some $0 < K_{i} < \infty$ and $\forall \xx \in \mathbb{R}^{p}$
\begin{subequations}
\label{strong_convex_boundedness}
\begin{equation}
0 \preceq \nabla^{2} f_{i}(\xx) \preceq K_{i} \mathbb{I}.
\label{convex_lipschitz}
\end{equation}
We also assume that $F$ is smooth and strongly convex, i.e., for some $0< \gamma \leq K < \infty$ and $\forall \xx \in  \mathbb{R}^{p}$ 
\begin{equation}
\label{F_strong_gen}
\gamma \mathbb{I} \preceq \nabla^{2} F(\xx) \preceq K \mathbb{I}. 
\end{equation}
\end{subequations}
Note that Assumption~\eqref{F_strong_gen} implies uniqueness of the minimizer, $\xx^{*}$, which is assumed to be attained. The quantity 
\begin{equation}
\label{cond_F}
\kappa \defeq \frac{K}{\gamma},
\end{equation}
is known as the condition number of the problem. 

For an integer $1 \leq q \leq n$, let $\mathcal{Q}$ be the set of indices corresponding to $q$ largest $K_{i}$'s and define the ``sub-sampling'' condition number as
\begin{equation}
\label{cond_q}
\kappa_{q} \defeq \frac{\widehat{K}_{q}}{\gamma},
\end{equation}
where 
\begin{equation}
\label{K_q}
\widehat{K}_{q} \defeq \frac{1}{q} \sum_{j \in \mathcal{Q}} K_{j}.
\end{equation}
It is easy to see that for any two integers $q$ and $r$ such that $1 \leq q \leq r \leq n$, we have $\kappa \leq \kappa_{r} \leq \kappa_{q}$. Finally, define
\begin{equation}
\tilde{\kappa} \defeq\begin{cases}
               \kappa_{1}, \quad \;\;\text{If sample $\mathcal{S}$ is drawn \textit{with} replacement}\\
               \kappa_{|\mathcal{S}|}, \quad \text{If sample $\mathcal{S}$ is drawn \textit{without} replacement}
            \end{cases},
\label{cond_w_wo_rep}
\end{equation}
where $\kappa_{1}$ and $\kappa_{|\mathcal{S}|}$ are as in~\eqref{cond_q}.

\subsection{Contributions}
\label{sec:contributions}
The contributions of this paper can be summarized as follows:
\begin{enumerate}[(1)]
	\item Under the above assumptions, we propose various globally convergent algorithms. These algorithms are guaranteed to approach the optimum, regardless of their starting point. Our algorithms are designed for the following settings.
	\begin{enumerate}[(i)]
		\item Algorithm~\ref{alg_global} practically implements~\eqref{sub_hess_unconstrained}. In other words, we incorporate sub-sampled Hessian while the full gradient is used. Theorem~\ref{global_newton} establishes the global convergence of Algorithm~\ref{alg_global}.
		
		\item Algorithms~\ref{alg_global_spectral} and~\ref{alg_global_ridge}, are modifications of Algorithm~\ref{alg_global} in which the sub-sampled Hessian is regularized by modifying its spectrum or by Levenberg-type regularization (henceforth called ridge-type regularization), respectively. Such regularization can be used to guarantee global convergence, in the absence of positive definiteness of the full Hessian. Theorems~\ref{global_newton_spectral} and~\ref{global_newton_ridge} as well as Corollaries~\ref{global_newton_spectral_2} and~\ref{global_newton_ridge_2} guarantee global convergence of these algorithms.
		
		\item Algorithm~\ref{alg_global_grad} is the implementation of the fully stochastic formulation~\eqref{sub_hess_grad_unconstrained}, in which the gradient as well the Hessian is sub-sampled. The global convergence of Algorithm~\ref{alg_global_grad} is guaranteed by Theorem~\ref{global_newton_grad}.
	\end{enumerate}
	\item For all of these algorithms, we give quantitative convergence results, i.e., our bounds contain an actual worst-case convergence rate. Our bounds here depend on problem dependent factors, i.e., condition numbers $\kappa$ and $\tilde{\kappa}$, and hold for a finite number of iterations. When these results are combined with those in SSN2~\cite{romassn2}, we obtain local convergence rates which are problem-independent; see Theorems~\ref{local_global_newton} and~\ref{local_global_newton_grad}. These connections guarantee that our proposed algorithms here, have a much faster convergence rates, at least locally, than what the simpler theorems in this paper suggest.
	
	\item For all of our algorithms, we present analysis for the case of inexact update where the linear system is solved only approximately, to a given tolerance. In addition, we establish criteria for the tolerance to guarantee faster convergence rate. The results of Theorems~\ref{global_newton_inexact},~\ref{global_newton_spectral},~\ref{global_newton_ridge}, and~\ref{global_newton_grad_inexact} give global convergence of the corresponding algorithms with inexact updates.
\end{enumerate}

\subsection{Related Work}
\label{sec:related_work}
The results of Section~\ref{sec:subsampl_newton_hess} offer computational efficiency for the regime where both $n$ and $p$ are large. However, it is required that $n$ is not so large as to make the gradient evaluation prohibitive. In such regime (where $n , p \gg 1$ but $n$ is not too large), similar results can be found in~\cite{byrd2011use, martens2010deep,pilanci2015newton,erdogdu2015convergence}. The pioneering work in~\cite{byrd2011use} establishes, for the first time, the convergence of Newton's method with sub-sampled Hessian and full gradient. There, two sub-sampled Hessian algorithms are proposed, where one is based on a matrix-free inexact Newton iteration and the other incorporates a preconditioned limited memory BFGS iteration. However, the results are asymptotic, i.e., for $k \rightarrow \infty$, and no quantitative convergence rate is given. In addition, convergence is established for the case where each $f_{i}$ is assumed to be strongly convex. Within the context of deep learning,~\cite{martens2010deep} is the first to study the application of a modification of Newton's method. It suggests a heuristic algorithm where at each iteration, the full Hessian is approximated by a relatively large subset of $\nabla^{2}f_{i}$'s, i.e. a ``mini-batch'', and the size of such mini-batch grows as the optimization progresses. The resulting matrix is then damped in a Levenberg-Marquardt style,~\cite{levenberg1944algorithm, marquardt1963algorithm}, and conjugate gradient,~\cite{shewchuk1994introduction}, is used to approximately solve the resulting linear system. The work in~\cite{pilanci2015newton} is the first to use ``sketching'' within the context of Newton-like methods. The authors propose a randomized second-order method which is based on performing an approximate Newton step using a randomly sketched Hessian. In addition to a few local convergence results, the authors give global convergence rate for self-concordant functions.  However, their algorithm is specialized to the cases where some square root of the Hessian matrix is readily available, i.e., some matrix $C(\xx) \in \mathbb{R}^{s \times p}$, such that $\nabla^{2} F(\xx) = C^{T}(\xx) C(\xx)$. Local convergence rate for the case where the Hessian is sub-sampled is first established in~\cite{erdogdu2015convergence}. The authors suggest an algorithm, where at each iteration, the spectrum of the sub-sampled Hessian is modified as a form of regularization and give locally linear convergence rate.

The results of Section~\ref{sec:subsampl_newton_hess_grad}, can be applied to more general setting where $n$ can be arbitrarily large. This is so since sub-sampling the gradient, in addition to that of the Hessian, allows for iteration complexity, which can be much smaller than $n$. Within the context of first order methods, there has been numerous variants of gradient sampling from a simple stochastic gradient descent,~\cite{robbins1951stochastic}, to the most recent improvements by incorporating the previous gradient directions in the current update~\cite{schmidt2013minimizing,senior2013empirical,bottou2010large,johnson2013accelerating}. For second order methods, such sub-sampling strategy has been successfully applied in large scale non-linear inverse problems~\cite{doas12,rodoas1,aravkin2012robust,van2013lost,haber2014simultaneous}. However, to the best of our knowledge, Section~\ref{sec:subsampl_newton_hess_grad} offers the first quantitative and global convergence results for such sub-sampled methods.

Finally, inexact updates have been used in many second-order optimization algorithms; see~\cite{byrd2013inexact,dembo1982inexact,scheinberg2013practical,nash2000survey,schmidt2009optimizing} and references therein.

\section{Sub-Sampling Hessian}
\label{sec:subsampl_newton_hess}
For the optimization problem~\eqref{obj}, at each iteration, consider picking a sample of indices from $\{1,2,\ldots,n\}$, uniformly at random \textit{with} or \textit{without} replacement. Let $\mathcal{S}$ and $|\mathcal{S}|$ denote the sample collection and its cardinality, respectively and define 
\begin{equation}
H(\xx) \defeq \frac{1}{|\mathcal{S}|} \sum_{j \in \mathcal{S}} \nabla^{2} f_{j}(\xx),
\label{subsampled_H}
\end{equation}
to be the sub-sampled Hessian. As mentioned before in Section~\ref{sec:general_background}, in order for such sub-sampling to be useful, we need to ensure that the sample size $|\mathcal{S}|$ satisfies the requirement~\hyperref[small_sample_size]{(R.1)}, while $H(\xx)$ is invertible as mentioned in~\hyperref[invertibility]{(R.2)}. Below, we make use of random matrix concentration inequalities to probabilistically guarantee such properties.

%In order to ensure that $H$ satisfies similar properties as in~\eqref{strong_convex_boundedness}, it is important that the sub-sampling is done such that the spectrum of the full Hessian, to the extend possible, is preserved. 
%\subsection{Uniform Sub-sampling Hessian - Matrix Chernoff}
%\label{uniform_sub-sampling_chernoff}
\begin{lemma}[Uniform Hessian Sub-Sampling]
Given any $0 < \epsilon < 1$, $0 < \delta < 1$, and $\xx \in \mathbb{R}^{p}$, 
if 
\begin{equation}
|\mathcal{S}| \geq \frac{2 \kappa_{1} \ln(p/\delta)}{\epsilon^{2}},
\label{uniform_sample_size_Chernoff}
\end{equation}
then for $H(\xx)$ defined in~\eqref{subsampled_H}, we have
\begin{equation}
\Pr\Big( (1-\epsilon)  \gamma \leq \lambda_{\min}\left(H(\xx) \right) \Big) \geq 1-\delta,
\label{eig_chernoff}
\end{equation}
where $\gamma$ and $\kappa_{1}$ are defined in~\eqref{F_strong_gen} and~\eqref{cond_q}, respectively.
\label{chernoff_lemma}
\end{lemma}
Hence, depending on $\kappa_{1}$, the sample size $|\mathcal{S}|$ can be smaller than $n$. In addition, we can always probabilistically guarantee that the sub-sampled Hessian is uniformly positive definite and, consequently, the direction given by it, indeed, yields a direction of descent.  

It is important to note that the sufficient sample size, $|\mathcal{S}|$, here grows only \textit{linearly} in $\kappa_{1}$, i.e., ${\Omega}(\kappa_{1})$, as opposed to \textit{quadratically}, i.e., $\Omega(\kappa_{1}^{2})$, in~\cite{romassn2,erdogdu2015convergence}. In fact, it might be worth elaborating more on the differences between the above sub-sampling strategy and that of SSN2~\cite[Lemmas 1, 2, and 3]{romassn2}. These differences, in fact, boil down to the differences between the requirement~\hyperref[invertibility]{(R.2)} and the corresponding one in SSN2~\cite[Section 1.1, (R.2)]{romassn2}. As a result of a ``coarser-grained'' analysis in the present paper and in order to guarantee global convergence, we only require that the sub-sampled Hessian is uniformly postive definite. Consequently, Lemma~\ref{chernoff_lemma} require a smaller sample size, i.e., in the order of $\kappa_{1}$ vs.\ $\kappa_{1}^{2}$ for SSN2~\cite[Lemma 1, 2, and 3]{romassn2}, while delivering a much weaker guarantee about the invertibility of the sub-sampled Hessian. In contrast, for the finer-grained analysis in SSN2~\cite{romassn2}, we needed a much stronger guarantee to preserve the spectrum of the true Hessian, and not just simple invertibility.

\subsection{Globally Convergent Newton with Hessian Sub-Sampling}
\label{sec:global_con_nm}
In this section, we give a globally convergent algorithms with Hessian sub-sampling which, starting from any initial iterate $\xx^{(0)} \in  \mathbb{R}^{p}$, converges to the optimal solution. Such algorithm for the unconstrained problem and when each $f_{i}$ is smooth and strongly convex is given in the pioneering work~\cite{byrd2011use}. Using Lemma~\ref{chernoff_lemma}, we now give such a globally-convergent algorithm under a milder assumption~\eqref{F_strong_gen}, where strong convexity is only assumed for $F$.

In Section~\ref{sec:global_con_nm_exact}, we first present an iterative algorithm in which, at every iteration, the linear system in~\eqref{sub_hess_unconstrained} is solved  exactly. In Section~\ref{sec:global_con_nm_inexact}, we then present a modified algorithm where such step is done only approximately and the update is computed inexactly, to within a desired tolerance. Finally, Section~\ref{sec:modified_hessian} will present algorithms in which the sub-sampled Hessian is regularized through modifying its spectrum or ridge-type regularization. For this latter section, the algorithms are given for the case of inexact update as extensions to exact solve is straightforward. The proofs of all the results are given in the appendix.

\subsubsection{Exact update}
\label{sec:global_con_nm_exact}
For the sub-sampled Hessian $H(\xx^{(k)})$, consider the update 
\begin{subequations}
\begin{equation}
\xx^{(k+1)} = \xx^{(k)} + \alpha_{k} \pp_{k},
\label{global_x}
\end{equation}
where
\begin{equation}
\pp_{k} = -[H(\xx^{(k)})]^{-1} \nabla F(\xx^{(k)}),
\label{global_p}
\end{equation} 
and
%\begin{equation}
%\begin{aligned}
%& \alpha_{k} = \arg\max_{\alpha \leq 1} \quad \alpha\\
%& \text{s.t.}
%\quad \alpha \in \mathcal{A},
%\end{aligned}
%\label{global_alpha}
%\end{equation}
%for  
%\begin{equation}
%\mathcal{A} \defeq \left\{ \alpha ; \; F(\xx^{(k)} + \alpha \pp_{k}) \leq F(\xx^{(k)}) + \alpha \beta \pp_{k}^{T} \nabla F(\xx^{(k)})\right\},
%\label{global_alpha_set}
%\end{equation}
\begin{equation}
\begin{aligned}
\alpha_{k} = \arg\max &\quad \alpha\\
\text{s.t.} & \quad \alpha \leq \widehat{\alpha} \\
& \quad F(\xx^{(k)} + \alpha \pp_{k}) \leq F(\xx^{(k)}) + \alpha \beta \pp_{k}^{T} \nabla F(\xx^{(k)}),
\end{aligned}
\label{global_alpha}
\end{equation}
\label{global_iteration}
\end{subequations}
for some $\beta \in (0,1)$ and $\widehat{\alpha} \geq 1$. Recall that~\eqref{global_alpha} can be approximately solved using various methods such as Armijo backtracking line search~\cite{armijo1966minimization}.

\begin{algorithm}
\caption{Globally Convergent Newton with Hessian Sub-Sampling}
\begin{algorithmic}[1]
\STATE \textbf{Input:} $\xx^{(0)}$, $0 < \delta < 1$, $0 < \epsilon < 1$, $0 < \beta < 1$, $\widehat{\alpha} \geq 1$ 
\STATE - Set the sample size, $|\mathcal{S}|$, with $\epsilon$ and $\delta$ as in~\eqref{uniform_sample_size_Chernoff} 
\FOR {$k = 0,1,2, \cdots$ until termination} 
\STATE - Select a sample set, $\mathcal{S}$, of size $|\mathcal{S}|$ and form $H(\xx^{(k)})$ as in~\eqref{subsampled_H}
\STATE - Update $\xx^{(k+1)}$ as in~\eqref{global_iteration} with $H(\xx^{(k)})$
\ENDFOR
\end{algorithmic}
\label{alg_global}
\end{algorithm}

\begin{theorem}[Global Convergence of Algorithm~\ref{alg_global}]
\label{global_newton}
Let Assumptions~\eqref{strong_convex_boundedness} hold. Using Algorithm~\ref{alg_global} with any $\xx^{(k)} \in  \mathbb{R}^{p}$,  with probability $1-\delta$, we have
%\begin{eqnarray*}
%F(\xx^{(k+1)}) - F(\xx^{(k)}) &\leq& - \frac{(1-\epsilon) \alpha_{k} \beta}{ K} \|\nabla F(\xx^{(k)})\|^{2}, \\
%F(\xx^{(k+1)}) - F(\xx^{*}) &\leq& (1 - \rho) \big(F(\xx^{(k)}) - F(\xx^{*})\big),
%\end{eqnarray*}  
\begin{equation}
F(\xx^{(k+1)}) - F(\xx^{*}) \leq (1 - \rho) \big(F(\xx^{(k)}) - F(\xx^{*})\big),
\label{global_linear_rate}
\end{equation}  
where $$\rho = \frac{2 \alpha_{k} \beta}{\tilde{\kappa}},$$ and $\tilde{\kappa}$ is defined as in~\eqref{cond_w_wo_rep}.
Moreover, the step size is at least
\begin{equation*}
\alpha_{k} \geq \frac{2 (1- \beta)(1-\epsilon)}{\kappa},
\end{equation*}
where $\kappa$ is defined as in~\eqref{cond_F}.
\end{theorem}

Theorem~\ref{global_newton} guarantees global convergence with at least a linear rate which depends on the quantities related to the specific problem. In SSN2~\cite{romassn2}, we have shown, through a finer grained analysis, that the locally linear convergence rate of such sub-sampled Newton method with a constant step size $\alpha_{k} = 1$ is indeed problem independent. In fact, it is possible to combine both results to obtain a globally convergent algorithm with a locally linear and \textit{problem-independent} rate, which is indeed much faster than what Theorem~\ref{global_newton} implies. 

%\begin{theorem}[Global Conv.\ of Alg.\ \ref{alg_global} with Problem-Independent Local Rate]
%Let Assumptions~\eqref{strong_convex_boundedness} hold and each $f_{i}$ have a Lipschitz continuous Hessian as 
%\begin{equation*}
%\| \nabla^{2} f_{i}(\xx) - \nabla^{2} f_{i}(\yy) \| \leq L \|\xx-\yy\|, \quad i = 1,2,\ldots,n.
%\end{equation*}
%Consider any $0 < \rho_{0} < \rho_{1} < 1/\sqrt{\kappa}$. Using Algorithm~\ref{alg_global} with any $\xx^{(0)} \in  \mathbb{R}^{p}$, $\widehat{\alpha} = 1$, $\epsilon \leq \rho_{0} / ( 4(1 + \rho_{0})\sqrt{\kappa_{1}})$ and $\beta \leq (1-\epsilon) ( 1 - \kappa \rho_{1}^{2} )/(2 \kappa )$, after 
%\begin{equation}
%k \geq \ln \left( \frac{2(1-\epsilon)^{2} \gamma^{3} (\rho_{1} - \rho_{0})^{2}}{L^{2} \left(F(\xx^{(0)}) - F(\xx^{*}) \right) } \right) / \ln\left(1 - \frac{4 \beta  (1- \beta)(1-\epsilon) }{\tilde{\kappa} \kappa}\right)
%\label{local_k}
%\end{equation}
%iterations, with probability $(1-\delta)^{k}$ we get ``problem-independent'' Q-linear convergence, i.e.,
%\begin{equation}
%\|\xx^{(k+1)} - \xx^{*}\| \leq \rho_{1} \|\xx^{(k)} - \xx^{*}\|,
%\label{local_rate}
%\end{equation}
%where $\kappa$, $\kappa_{1}$ and $\tilde{\kappa}$ are defined in~\eqref{cond_F},~\eqref{cond_q} and~\eqref{cond_w_wo_rep}, respectively. 
%Moreover, the step size of $\alpha_{k} = 1$ is selected in~\eqref{global_alpha} for all subsequent iterations.
%\label{local_global_newton}
%\end{theorem}

\begin{theorem}[Global Conv.\ of Alg.\ \ref{alg_global} with Problem-Independent Local Rate]
Let Assumptions~\eqref{strong_convex_boundedness} hold and each $f_{i}$ have a Lipschitz continuous Hessian as 
\begin{equation}
\label{Hess_Lip}
\| \nabla^{2} f_{i}(\xx) - \nabla^{2} f_{i}(\yy) \| \leq L \|\xx-\yy\|, \quad i = 1,2,\ldots,n.
\end{equation}
Consider any $0 < \rho_{0} < \rho_{1} < 1$. Using Algorithm~\ref{alg_global} with any $\xx^{(0)} \in  \mathbb{R}^{p}$, $\widehat{\alpha} = 1$, $0 < \beta < 1/2$ and $$\epsilon \leq \min\left\{\frac{(1-2\beta)}{2(1-\beta)},\frac{\rho_{0}}{4(1 + \rho_{0})\sqrt{\kappa_{1}}}\right\},$$ after 
\begin{equation}
k \geq \ln \left( \frac{2(1-\epsilon)^{2} \gamma^{4} (\rho_{1} - \rho_{0})^{2}\big(1-2\epsilon - 2 (1-\epsilon) \beta\big)^{2}}{K L^{2} \left(F(\xx^{(0)}) - F(\xx^{*}) \right) } \right) / \ln\left(1 - \frac{4 \beta  (1- \beta)(1-\epsilon) }{\tilde{\kappa} \kappa}\right)
\label{local_k}
\end{equation}
iterations, with probability $(1-\delta)^{k}$ we get ``problem-independent'' Q-linear convergence, i.e.,
\begin{equation}
\|\xx^{(k+1)} - \xx^{*}\| \leq \rho_{1} \|\xx^{(k)} - \xx^{*}\|,
\label{local_rate}
\end{equation}
where $\kappa$, $\kappa_{1}$ and $\tilde{\kappa}$ are defined in~\eqref{cond_F},~\eqref{cond_q} and~\eqref{cond_w_wo_rep}, respectively. 
Moreover, the step size of $\alpha_{k} = 1$ is selected in~\eqref{global_alpha} for all subsequent iterations.
\label{local_global_newton}
\end{theorem}

%For example, let $x^{(k)}$ be close enough to $\xx^{*}$, as described in SSN2~\cite[Theorem 2]{romassn2}. Armijo rule is first tested with $\alpha_{k} = 1$, we get
%\begin{equation*}
%\|\xx^{(k+1)} - \xx^{*}\| \leq \rho \|\xx^{(k)} - \xx^{*}\|,
%\end{equation*}
%which implies that
%\begin{equation*}
%F(\xx^{(k+1)}) - F(\xx^{*}) \leq \kappa \rho^{2} \left( F(\xx^{(k)}) - F(\xx^{*}) \right).
%\end{equation*}
%\begin{equation*}
%F(\xx^{(k+1)}) \leq \kappa \rho^{2} F(\xx^{(k)}) + (1-\kappa \rho^{2}) F(\xx^{*}).
%\end{equation*}
%\begin{equation*}
%F(\xx^{(k+1)}) \leq \kappa \rho^{2} F(\xx^{(k)}) + (1-\kappa \rho^{2}) F(\xx^{*}).
%\end{equation*}
%Hence, we get
%\begin{equation*}
%F(\xx^{(k+1)}) - F(\xx^{(k)}) \leq -( 1 - \kappa \rho^{2} ) \left( F(\xx^{(k)}) - F(\xx^{*}) \right).
%\end{equation*}
%Now by Lipschitz continuity (Theorem 2.1.5 Nesterov)
%\begin{equation*}
%F(\xx^{(k)}) - F(\xx^{*}) \geq \frac{1}{2 K } \| \nabla F(\xx^{(k)}) \|^{2},
%\end{equation*}
%hence
%\begin{equation*}
%F(\xx^{(k+1)}) - F(\xx^{(k)}) \leq - \frac{( 1 - \kappa \rho^{2} )}{2 K } \| \nabla F(\xx^{(k)}) \|^{2}.
%\end{equation*}
%Now we want,
%\begin{equation*}
%\frac{( 1 - \kappa \rho^{2} )}{2 K } \geq \frac{\beta}{\mu} 
%\end{equation*}
%or
%\begin{equation*}
%\beta \leq \frac{( 1 - \kappa \rho^{2} )}{2 \kappa }.
%\end{equation*}
%Because then since
%\begin{equation*}
%- \frac{\beta}{\mu} \| \nabla F(\xx^{(k)}) \|^{2} \leq -\beta \nabla F(\xx^{(k)})^{T} [H(\xx^{(k)})]^{-1} \nabla F(\xx^{(k)}) = -\beta \pp_{k}^{T} \nabla F(\xx^{(k)}) ,
%\end{equation*}
%and so the step-size $\alpha_{k}=1$ will be accepted.

In fact, it is even possible to obtain a globally convergent algorithm with locally superlinear rate of convergence using Algorithm~\ref{alg_global} with iteration dependent $\hat{\epsilon}^{(k)}$ as 
\begin{eqnarray*}
\hat{\epsilon}^{(k)} \leq \frac{\epsilon^{(k)}}{4\sqrt{\kappa_{1}}} ,
\end{eqnarray*}
where $\epsilon^{(k)}$ is chosen as in SSN2~\cite[Theorem 3 or 4]{romassn2}. The details are similar to Theorem~\ref{local_global_newton} and are omitted here.

%In particular, Algorithm~\ref{alg_global} is similar to SSN2~\cite[Algorithms 1 and 2]{romassn2}, except that the step size $\alpha_{k}$ is chosen by Armijo rule to ensure sufficient decrease of the objective function, when far from the solution. Consequently, using SSN2~\cite[Theorems 2, 3, or 4]{romassn2}, it is guaranteed that, when close enough to $\xx^{*}$, the algorithm adopts the natural Newton step size of $\alpha_{k} = 1$.

\subsubsection{Inexact update}
\label{sec:global_con_nm_inexact}
%In many situations, the solution of such linear system can indeed be the bottleneck of computations. Solving such linear system approximately can offer additional computational efficiency per iteration. As such, it is imperative to have algorithms that can incorporate such inexact search direction and still obtain convergence. For such inexact updates, we give global convergence results.
In many situations, where finding the exact update, $\pp_{k}$, in~\eqref{global_p} is computationally expensive, it is imperative to be able to calculate the update direction $\pp_{k}$ only approximately. Such inexactness can indeed reduce the computational costs of each iteration and is particularly beneficial when the iterates are far from the optimum. This makes intuitive sense because, if the current iterate is far from $\xx^{*}$, it may be computationally wasteful to exactly solve for $\pp_{k}$ in~\eqref{global_p}. Such inexact updates have been used in many second-order optimization algorithms, e.g.~\cite{byrd2013inexact,dembo1982inexact,scheinberg2013practical}. Here, in the context of uniform sub-sampling, we give similar global results using inexact search directions inspired by~\cite{byrd2013inexact} . 

For computing the search direction, $\pp_{k}$, consider the linear system $H(\xx^{(k)}) \pp_{k} = -\nabla F(\xx^{(k)})$ at $k^{th}$ iteration. Instead, in order to allow for inexactness,  one can solve the linear system such that for some $0 \leq \theta_{1},\theta_{2} < 1$, $\pp_{k}$ satisfies
\begin{subequations}
\begin{eqnarray}
&&\|H(\xx^{(k)})\pp_{k} + \nabla F(\xx^{(k)})\| \leq \theta_{1} \|\nabla F(\xx^{(k)})\|, \label{global_p_inexact}\\
&&\pp_{k}^{T} \nabla F(\xx^{(k)}) \leq -(1-\theta_{2}) \pp_{k}^{T} H(\xx^{(k)})\pp_{k}. \label{angle_inexact}
\end{eqnarray}
\label{global_update_inexact}
\end{subequations}
The condition~\eqref{global_p_inexact} is the usual relative residual of the approximate solution. However, for a given $\theta_{1}$, any $\pp_{k}$ satisfying~\eqref{global_p_inexact} might not necessarily result in a descent direction. As a result, condition~\eqref{angle_inexact} ensures that such a $\pp_{k}$ is always a direction of descent. Note that given any $0 \leq \theta_{1},\theta_{2} < 1$, one can always find a $\pp_{k}$ satisfying~\eqref{global_update_inexact} (e.g., the exact solution always satisfies~\eqref{global_update_inexact}). 

\begin{theorem}[Global Convergence of Algorithm~\ref{alg_global}: Inexact Update]
\label{global_newton_inexact}
Let Assumptions~\eqref{strong_convex_boundedness} hold. Also let $0 \leq \theta_{1} < 1$ and $0 \leq \theta_{2} < 1$ be given. Using Algorithm~\ref{alg_global} with any $\xx^{(k)} \in  \mathbb{R}^{p}$, and the ``inexact'' update direction~\eqref{global_update_inexact} instead of~\eqref{global_p}, with probability $1-\delta$, we have that~\eqref{global_linear_rate} holds where
\begin{enumerate}[(i)]
	\item if $$\theta_{1} \leq \sqrt{\frac{(1-\epsilon)}{4 \tilde{\kappa}}},$$ then 
$\rho = {\alpha_{k} \beta }/{\tilde{\kappa}}$, 
\item otherwise $
\rho = {2(1 -\theta_{2}) (1 -\theta_{1})^{2} (1-\epsilon) \alpha_{k} \beta}/{\tilde{\kappa}^{2}}$,
\end{enumerate}
with $\tilde{\kappa}$ defined as in~\eqref{cond_w_wo_rep}.
Moreover, for both cases, the step size is at least
\begin{equation*}
\alpha_{k} \geq \frac{2(1-\theta_{2})(1- \beta)(1-\epsilon)}{\kappa},
\end{equation*}
where $\kappa$ is defined as in~\eqref{cond_F}.
\end{theorem}

\comment Theorem~\ref{global_newton_inexact} indicates that, in order to guarantee a faster convergence rate, the linear system needs to be solved to a ``small-enough'' accuracy, which is in the order of $\mathcal{O}(\sqrt{1/\tilde{\kappa}})$. In other words, the degree of accuracy inversely depends on the square root of the sub-sampling condition number and the larger the condition number, $\tilde{\kappa}$, the more accurately we need to solve the linear system. However, it is interesting to note that using a tolerance of order $\mathcal{O}(\sqrt{1/\tilde{\kappa}})$, we can still guarantee a similar global convergence rate as that of the algorithm with exact updates!

\comment The dependence of the guarantees of Theorem~\ref{global_newton_inexact} on the inexactness parameters, $\theta_{1}$ and $\theta_{2}$, is indeed intuitive. The minimum amount of decrease in the objective function is mainly dependent on $\theta_{1}$, i.e., the accuracy of the the linear system solve. On the other hand, the dependence of the step size, $\alpha_{k}$, on $\theta_{2}$ indicates that the algorithm can take larger steps along a search direction, $\pp_{k},$ that points more accurately towards the direction of the largest rate of decrease, i.e., $\pp_{k}^{T} \nabla F(\xx^{(k)})$ is more negative which means that $\pp_{k}$ points more accurately in the direction of $-\nabla F(\xx^{(k)})$. As a result, the more exactly we solve for $\pp_{k}$, the larger the step that the algorithm might take and the larger the decrease in the objective function. However, the cost of any such calculation at each iteration might be expensive. As a result, there is a trade-off between accuracy for computing the update in each iteration and the overall progress of the algorithm.

\subsection{Modifying the Sample Hessian}
\label{sec:modified_hessian}
As mentioned in the introduction, if $F$ is not strongly convex, it is still possible to obtain globally convergent algorithms. This is done through regularization of the Hessian of $F$ to ensure that the resulting search direction is indeed a direction of descent. Even when $F$ is strongly convex, the use of regularization can still be beneficial. Indeed, the lower bound for the step size in Theorems~\ref{global_newton} and~\ref{global_newton_inexact} imply that a small $\gamma$ can potentially slow down the \textit{initial} progress of the algorithm. More specifically, small $\gamma$ might force the algorithm to adopt small step sizes, at least in the early stages of the algorithm. This observation is indeed reinforced by the composite nature of error recursions obtained in SSN2~\cite{romassn2}. In particular, it has been shown in SSN2~\cite{romassn2} that in the early stages of the algorithm, when the iterates are far from the optimum, the error recursion is dominated by a quadratic term which transitions to a linear term as the iterates get closer to the optimum. However, unlike the linear term, the quadratic term is negatively affected by the small values of $\gamma$. In other words, small $\gamma$ can hinder the initial progress and even using the full Hessian cannot address this issue. As a result, one might resort to regularization of the (estimated) Hessian to improve upon the initial slow progress. Here, we explore two strategies for such regularization which are incorporated as part of our algorithms. The results are given for when~\eqref{global_p} is solved approximately with a ``small-enough'' tolerance and for the case of sub-sampling without replacement. Extensions to arbitrary tolerance as well as sampling with replacement is as before and straightforward. 

\subsubsection{Spectral Regularization}
\label{sec:spectral_reg}
In this section, we follow the ideas presented in~\cite{erdogdu2015convergence}, by accounting for such a potentially negative factor $\gamma$, through a regularization of eigenvalue distribution of the sub-sampled Hessian. 
More specifically, for some $\lambda \geq 0 $, let
\begin{subequations}
\begin{equation}
\hat{H} \defeq \mathcal{P}(\lambda; H),
\end{equation}
where $\mathcal{P}(\lambda; H)$ is an operator which is defined as
\begin{equation}
\mathcal{P}(\lambda; H) \defeq \lambda \mathbb{I} + \arg\max_{ X \succeq 0} \|H - \lambda \mathbb{I} - X\|_{F}.
\end{equation}
\label{spectral_proj}
\end{subequations}
The operation~\eqref{spectral_proj} can be equivalently represented as 
\begin{equation*}
\mathcal{P}(\lambda; H) = \sum_{i=1}^{p} \max \left \{ \lambda_{i}(H), \lambda \right\} \vv_{i} \vv_{i}^{T},
\end{equation*}
with $\vv_{i}$ being the $i^{th}$ eigenvector of $H$ corresponding to the $i^{th}$ eigenvalue, $\lambda_{i}(H)$. 
Operation~\eqref{spectral_proj} can be performed using truncated SVD (TSVD). Note that although TSVD can be done through standard methods, faster randomized alternatives exist which provide accurate approximations to TSVD much more efficiently~\cite{halko2011finding}. 

%
%where $\mathcal{P}(\lambda; H)$ is an operator which is defined as
%\begin{equation}
%\lambda_{i}(\mathcal{P}(\lambda; H)) = 
	%\left\{
	%\begin{array}{ll}
		%\lambda_{i}(H) \quad \text{ If  } \lambda_{i}(H) \geq \lambda\\
		%\lambda \quad \quad \quad \text{ Otherwise }	
	%\end{array}
  %\right., \quad i = 1,2,\ldots,p,
%\end{equation}
%\label{spectral_proj}
%\end{subequations}
%where $\lambda_{i}(A)$ denote the $i^{th}$ eigenvalue of the matrix $A$.

\begin{algorithm}
\caption{Globally Convergent Newton with Hessian Sub-Sampling and Spectral Regularization}
\begin{algorithmic}[1]
\STATE \textbf{Input:} $\xx^{(0)}$, $0 < \beta < 1$, $\widehat{\alpha} \geq 1$ 
\STATE - Set any sample size, $|\mathcal{S}| \geq 1$ 
\FOR {$k = 0,1,2, \cdots$ until termination} 
\STATE - Select a sample set, $\mathcal{S}$, of size $|\mathcal{S}|$ and form $H(\xx^{(k)})$ as in~\eqref{subsampled_H}
\STATE - Compute $\lambda_{\min}\left( H(\xx^{(k)}) \right)$ \label{truncated_SVD_1}
\STATE - Set $\lambda^{(k)} > \lambda_{\min}\left( H(\xx^{(k)}) \right)$ \label{truncated_SVD_2}
\STATE - $\hat{H}(\xx^{(k)})$ as in~\eqref{spectral_proj} with $\lambda$ \label{truncated_SVD_3}
\STATE - Update $\xx^{(k+1)}$ as in~\eqref{global_iteration} with $\hat{H}(\xx^{(k)})$ \label{truncated_SVD_4}
\ENDFOR
\end{algorithmic}
\label{alg_global_spectral}
\end{algorithm}

As a result of Steps~\ref{truncated_SVD_1}--\ref{truncated_SVD_3} in Algorithm~\ref{alg_global_spectral}, the regularized sub-sampled matrix, $\hat{H}(\xx^{(k)})$, is always positive definite with minimum eigenvalue $\lambda^{(k)} > 0$. Consequently, just to obtain global convergence, there is no need to resort to sampling Lemma~\ref{chernoff_lemma} to ensure invertibility of the sub-sampled matrix. In fact, such regularization guarantees the global convergence of Algorithm~\ref{alg_global_spectral}, even in the absence of strong convexity assumption~\eqref{F_strong_gen}. Theorem~\ref{global_newton_spectral} gives such a result for Algorithm~\ref{alg_global_spectral} in the case of inexact update. %Note that, in this case, the inexactness tolerance, $\theta_{1}^{(k)}$, is dependent on the regularization threshold $\lambda^{(k)}$.  

\begin{theorem}[Global Convergence of Algorithm~\ref{alg_global_spectral}: Arbitrary Sample Size]
\label{global_newton_spectral}
Let Assumption~\eqref{convex_lipschitz} hold and $0 < \theta_{2} < 1$ be given. Using Algorithm~\ref{alg_global_spectral} with sampling without replacement, for any $\xx^{(k)} \in  \mathbb{R}^{p}$ and the ``inexact'' update direction~\eqref{global_update_inexact} instead of~\eqref{global_p}, if 
	\begin{equation*}
\theta_{1}^{(k)} \leq \hf \sqrt{\frac{\lambda^{(k)}}{\max \{\lambda^{(k)},\widehat{K}_{|\mathcal{S}|} \}}},
\end{equation*}
we have 
\begin{equation*}
F(\xx^{(k+1)}) \leq F(\xx^{(k)}) - \frac{\alpha_{k} \beta}{2 \max\left\{ \widehat{K}_{|\mathcal{S}|}, \lambda^{(k)} \right\}} \|\nabla F(\xx^{(k)})\|^{2},
\end{equation*}
where $\widehat{K}_{|\mathcal{S}|}$ is defined as in~\eqref{K_q}.  
If in addition, Assumption~\eqref{F_strong_gen} holds, then we have~\eqref{global_linear_rate} with 
\begin{equation*}
\rho = \frac{\alpha_{k} \beta \gamma }{\max\left\{ \widehat{K}_{|\mathcal{S}|}, \lambda^{(k)} \right\}}.
\end{equation*}
Moreover, for both cases, the step size is at least
\begin{equation*}
\alpha_{k} \geq \frac{2(1-\theta_{2}) (1-\beta) \lambda^{(k)} }{K}.
\end{equation*}
\end{theorem}

If Assumption~\eqref{F_strong_gen} holds, the main issue with using arbitrary sample size is that if $|\mathcal{S}|$ is not large enough, then $H(\xx^{(k)})$ might be (nearly) singular, i.e., $\lambda_{\min} (H(\xx^{(k)})) \approx 0$. This issue will in turn necessitate a heavier regularization, i.e., larger $\lambda^{(k)}$. Otherwise, having $\lambda^{(k)} \ll 1$ implies a more accurate update, i.e., smaller $\theta_{1}^{(k)}$, and/or a smaller step-size (see the upper bound for $\theta_{1}^{(k)}$ and the lower bound for $\alpha_{k}$ in Theorem~\ref{global_newton_spectral}). As a result, if Assumption~\eqref{F_strong_gen} holds, it might be beneficial to use Lemma~\ref{chernoff_lemma} to, at least, guarantee that the sufficient upper bound for $\theta_{1}$ is always bounded away from zero and minimum step-size is independent of regularization. Indeed, in Corollary~\ref{global_newton_spectral_2} since $\lambda^{(k)} > (1-\epsilon) \gamma$, the sufficient upper bound for the inexactness tolerance, $\theta_{1}^{(k)}$, is always larger than $0.5\sqrt{(1-\epsilon) /\tilde{\kappa}}$.

\begin{corollary}[Global Convergence of Algorithm~\ref{alg_global_spectral}: Sampling as in Lemma~\ref{chernoff_lemma}]
\label{global_newton_spectral_2}
Let Assumptions~\eqref{strong_convex_boundedness} hold. Also let $0 < \epsilon < 1$, $0 < \delta < 1$ and $0 < \theta_{2} < 1$ be given. For any $\xx^{(k)} \in  \mathbb{R}^{p}$, using Algorithm~\ref{alg_global_spectral} with $\mathcal{S}$ chosen without replacement as in Lemma~\ref{chernoff_lemma}, and the ``inexact'' update direction~\eqref{global_update_inexact} instead of~\eqref{global_p}, if 
	\begin{equation*}
\theta_{1}^{(k)} \leq \hf \sqrt{\frac{\lambda^{(k)}}{\max \{\lambda^{(k)},\widehat{K}_{|\mathcal{S}|} \}}},
\end{equation*} 
we have~\eqref{global_linear_rate} with 
\begin{equation*}
\rho = \frac{\alpha_{k} \beta \gamma }{\max\left\{ \widehat{K}_{|\mathcal{S}|}, \lambda^{(k)} \right\}},
\end{equation*}
where $\widehat{K}_{|\mathcal{S}|}$ is defined as in~\eqref{K_q}.
Moreover, with probability $1-\delta$, the step size is at least
\begin{equation*}
\alpha_{k} \geq \frac{2(1-\theta_{2}) (1-\beta) (1-\epsilon)}{\kappa}.
\end{equation*}
\end{corollary}

\subsubsection{Ridge Regularization}
\label{sec:ridge}
As an alternative to the spectral regularization, we can consider the following simple ridge-type regularization 
\begin{equation}
\hat{H}(\xx) \defeq H(\xx) + \lambda \mathbb{I},
\label{ridge}
\end{equation}
for some $\lambda \geq 0$, similar to  the Levenberg-Marquardt type algorithms~\cite{marquardt1963algorithm}. Such regularization might be preferable to the spectral regularization of Section~\ref{sec:spectral_reg}, as it avoids the projection operation~\ref{spectral_proj} at every iteration. Theorem~\ref{global_newton_ridge} gives a global convergence guarantee for Algorithm~\ref{alg_global_ridge} in the case of inexact update.

\begin{algorithm}
\caption{Globally Convergent Newton with Hessian Sub-Sampling and Ridge Regularization}
\begin{algorithmic}[1]
\STATE \textbf{Input:} $\xx^{(0)}$, $0 < \beta < 1$, $\widehat{\alpha} \geq 1$, $\lambda > 0$
\STATE - Set any sample size, $|\mathcal{S}| \geq 1$
\FOR {$k = 0,1,2, \cdots$ until termination} 
\STATE - Select a sample set, $\mathcal{S}$, of size $|\mathcal{S}|$ and form $H(\xx^{(k)})$ as in~\eqref{subsampled_H}
\STATE - $\hat{H}(\xx^{(k)})$ as in~\eqref{ridge} with $\lambda$
\STATE - Update $\xx^{(k+1)}$ as in~\eqref{global_iteration} with $\hat{H}(\xx^{(k)})$
\ENDFOR
\end{algorithmic}
\label{alg_global_ridge}
\end{algorithm}

\begin{theorem}[Global Convergence of Algorithm~\ref{alg_global_ridge}: Arbitrary Sample Size]
\label{global_newton_ridge}
Let Assumption~\eqref{convex_lipschitz} hold and $0 < \theta_{2} < 1$ be given. Using Algorithm~\ref{alg_global_ridge} with sampling without replacement, for any $\xx^{(k)} \in  \mathbb{R}^{p}$ and the ``inexact'' update direction~\eqref{global_update_inexact} instead of~\eqref{global_p}, if 
	\begin{equation*}
\theta_{1} \leq \hf \sqrt{\frac{\lambda}{K+\lambda}},
\end{equation*}
then
\begin{equation*}
F(\xx^{(k+1)}) \leq F(\xx^{(k)}) - \frac{\alpha_{k} \beta}{2\left( \widehat{K}_{|\mathcal{S}|} + \lambda \right)} \|\nabla F(\xx^{(k)})\|^{2},
\end{equation*}  
where $\widehat{K}_{|\mathcal{S}|}$ is defined as in~\eqref{K_q}.
If, in addition, Assumption~\eqref{F_strong_gen} holds, then we have~\eqref{global_linear_rate} with 
\begin{equation*}
\rho = \frac{\alpha_{k} \beta \gamma}{ \widehat{K}_{|\mathcal{S}|} + \lambda}.
\end{equation*}
Moreover, for both cases, the step size is at least
\begin{equation*}
\alpha_{k} \geq \frac{2(1-\theta_{2})(1-\beta) \lambda }{K}.
\end{equation*}
\end{theorem}

As mentioned before in Section~\ref{sec:spectral_reg}, under Assumption~\eqref{F_strong_gen}, if the sample size $|\mathcal{S}|$ is not chosen large enough then $H(\xx^{(k)})$ might (nearly) be singular. This can in turn cause a need for a larger $\lambda$ or, alternatively, if $\lambda \ll 1$, the accuracy tolerance $\theta_{1}$ and the step-size $\alpha_{k}$ can be very small. However, by using the sample size given by Lemma~\ref{chernoff_lemma}, one can probabilistically guarantee a minimum step-size as well as a minimum sufficient upper bound for $\theta_{1}$ which are bounded away from zero even if $\lambda = 0$.

\begin{corollary}[Global Convergence of Algorithm~\ref{alg_global_ridge}: Sampling as in Lemma~\ref{chernoff_lemma}]
\label{global_newton_ridge_2}
Let Assumptions~\eqref{strong_convex_boundedness} hold. Also let $0 < \epsilon < 1$, $0 < \delta < 1$ and $0 < \theta_{2} < 1$ be given. Using Algorithm~\ref{alg_global_ridge} with $\mathcal{S}$ chosen without replacement as in Lemma~\ref{chernoff_lemma}, for any $\xx^{(k)} \in  \mathbb{R}^{p}$ and the ``inexact'' update direction~\eqref{global_update_inexact} instead of~\eqref{global_p}, if 
\begin{equation*}
\theta_{1} \leq \hf \sqrt{\frac{(1-\epsilon)\gamma + \lambda}{\widehat{K}_{|\mathcal{S}|} + \lambda}},
\end{equation*} 
we have~\eqref{global_linear_rate} with 
\begin{equation*}
\rho = \frac{\alpha_{k} \beta \gamma }{\widehat{K}_{|\mathcal{S}|} +\lambda },
\end{equation*}
where $\widehat{K}_{|\mathcal{S}|}$ is defined as in~\eqref{K_q}.
Moreover, with probability $1-\delta$, the step size is at least
\begin{equation*}
\alpha_{k} \geq \frac{2(1-\theta_{2}) (1-\beta) ((1-\epsilon)\gamma+\lambda)}{K}.
\end{equation*}
\end{corollary}

\comment In both regularization methods of Sections~\ref{sec:spectral_reg} and~\ref{sec:ridge} , as the parameter $\lambda$ gets larger, the methods behaves more like gradient descent. For example, using the regularization of Section~\ref{sec:spectral_reg}, for $\lambda = \widehat{K}_{|\mathcal{S}|}$, we have that $\hat{H}(\xx^{(k)}) = \widehat{K}_{|\mathcal{S}|} \mathbb{I}$, and the method is exactly gradient descent. As a result, a method with heavier regularization might not benefit from more accurate sub-sampling. Hence, at early stages of the iterations, it might be better to have small sample size with heavier regularization, while as the iterations get closer to the optimum, larger sample size with lighter regularization might be beneficial. 

\section{Sub-Sampling Hessian \& Gradient}
\label{sec:subsampl_newton_hess_grad}
In order to compute the update $\xx^{(k+1)}$ in Section~\ref{sec:subsampl_newton_hess}, full gradient was used. In many problems, this can be a major bottleneck and reduction in computational costs can be made by considering sub-sampling the gradient as well. This issue arises more prominently in high dimensional settings where $n , p \gg 1$ and evaluating the full gradient at each iteration can pose a significant challenge. In such problems, sub-sampling the gradient can, at times, drastically reduce the computational complexity of many problems. 

Consider selecting a sample collection from $\{1,2,\ldots,n\}$, uniformly at random with replacement and let 
\begin{equation}
\bgg(\xx) \defeq \frac{1}{|\mathcal{S}|} \sum_{j \in \mathcal{S}} \nabla f_{j}(\xx),
\label{subsampled_G}
\end{equation}
be the sub-sampled gradient. As mentioned before in the requirement~\hyperref[invertibility]{(R.2)}, we need to ensure that sampling is done in a way to keep as much of the first order information from the full gradient as possible. 

By a simple observation, the gradient $\nabla F(\xx)$ can be written in \textit{matrix-matrix product} from as 
\begin{eqnarray*}
\nabla F(\xx) = \begin{pmatrix}
\mid & \mid & & \mid \\
\nabla f_{1}(\xx) & \nabla f_{2}(\xx) & \cdots & \nabla f_{n}(\xx)\\
\mid & \mid & & \mid \\
\end{pmatrix} \begin{pmatrix}1/n\\1/n \\ \vdots \\ 1/n\end{pmatrix}.
\end{eqnarray*}
Hence, we can use approximate matrix multiplication results as a fundamental primitive in RandNLA~\cite{mahoney2011randomized,drineas2006fast}, to probabilistically control the error in approximation of $\nabla F(\xx)$ by $\bgg(\xx)$, through uniform sampling of the columns and rows of the involved matrices above. As a result, we have the following lemma (a more general form of this lemma for the case of constrained optimization is given in the companion paper, SSN2~\cite[Lemma 4]{romassn2}). 

\begin{lemma}[Uniform Gradient Sub-Sampling]
\label{randnla_lemma}
For a given $\xx \in \mathbb{R}^{p}$, let 
\begin{equation*}
\|\nabla f_{i}(\xx) \| \leq G(\xx), \quad i =1,2,\ldots,n.
\end{equation*} 
For any $0 < \epsilon < 1$ and $0 < \delta < 1$, if
\begin{equation}
|\mathcal{S}| \geq \frac{G(\xx)^{2}}{\epsilon^{2}} \Big( 1 + \sqrt{8 \ln \frac{1}{\delta}}\Big)^{2},
\label{uniform_sample_size_gradient}
\end{equation}
then for $\bgg(\xx)$ defined in~\eqref{subsampled_G}, we have
\begin{equation*}
\Pr \Big( \|\nabla F(\xx) - \bgg(\xx) \| \leq \epsilon \Big) \geq 1-\delta.
\end{equation*}
\end{lemma}

\comment In order to use the above result in our gradient sub-sampling, we need to be able to efficiently estimate $G(\xx)$ at every iteration. Fortunately, in many different problems, this is often possible; for concrete examples, see Section~\ref{sec:examples}.

\subsection{Globally Convergent Newton with Gradient and Hessian Sub-Sampling}
\label{sec:global_con_nm_grad}
We now show that by combining the gradient sub-sampling of Lemma~\ref{randnla_lemma} with Hessian sub-sampling of Lemma~\ref{chernoff_lemma}, we can still obtain global guarantees for some modification of Algorithm~\ref{alg_global}. This indeed generalizes our algorithms to fully stochastic variants where both the gradient and the Hessian are approximated.

In Section~\ref{sec:global_con_nm_exact_grad}, we first present an iterative algorithm with exact update where, at every iteration, the linear system in~\eqref{sub_hess_unconstrained} is solved  exactly. In Section~\ref{sec:global_con_nm_inexact_grad}, we then present a modified algorithm where such step is done only approximately and the update is computed inexactly, to within a desired tolerance. The proofs of all the results are given in the appendix.

\subsubsection{Exact update}
\label{sec:global_con_nm_exact_grad}
For the sub-sampled Hessian, $H(\xx^{(k)})$, and the sub-sampled gradient, $\bgg(\xx^{(k)})$, consider the update 
\begin{subequations}
\begin{equation}
\xx^{(k+1)} = \xx^{(k)} + \alpha_{k} \pp_{k},
\label{global_x_grad}
\end{equation}
where
\begin{equation}
\pp_{k} = -[H(\xx^{(k)})]^{-1}\bgg(\xx^{(k)}),
\label{global_p_grad}
\end{equation} 
and
\begin{equation}
\begin{aligned}
\alpha_{k} = \arg \max &\quad \alpha\\
\text{s.t.} & \quad \alpha \leq \widehat{\alpha} \\
& \quad F(\xx^{(k)} + \alpha \pp_{k}) \leq F(\xx^{(k)}) + \alpha \beta \pp_{k}^{T} \bgg(\xx^{(k)}),
\end{aligned}
\label{global_alpha_grad}
\end{equation}
for some $\beta \in (0,1)$ and $\widehat{\alpha} \geq 1$.
\label{global_iteration_grad}
\end{subequations}

\begin{algorithm}
\caption{Globally Convergent Newton with Hessian and Gradient Sub-Sampling}
\begin{algorithmic}[1]
\STATE \textbf{Input:} $\xx^{(0)}$, $0 < \delta < 1$, $0 < \epsilon_{1} < 1$, $0 < \epsilon_{2} < 1$, $0 < \beta < 1$, $\widehat{\alpha} \geq 1$ and $\sigma \geq 0$
\STATE - Set the sample size, $|\mathcal{S}_{H}|$, with $\epsilon_{1}$ and $\delta$ as in~\eqref{uniform_sample_size_Chernoff}
\FOR {$k = 0,1,2, \cdots$ until termination} 
\STATE - Select a sample set, $\mathcal{S}_{H}$, of size $|\mathcal{S}_{H}|$ and form $H(\xx^{(k)})$ as in~\eqref{subsampled_H}
\STATE - Set the sample size, $|\mathcal{S}_{\bgg}|$, with $\epsilon_{2}$, $\delta$ and $\xx^{(k)}$ as in~\eqref{uniform_sample_size_gradient}
\STATE - Select a sample set, $\mathcal{S}_{\bgg}$ of size $|\mathcal{S}_{\bgg}|$ and form $\bgg(\xx^{(k)})$ as in~\eqref{subsampled_G}
\IF {$\|\bgg(\xx^{(k)})\| < \sigma \epsilon_{2}$}
\STATE- STOP
\ENDIF
\STATE - Update $\xx^{(k+1)}$ as in~\eqref{global_iteration_grad} with $H(\xx^{(k)})$ and $\bgg(\xx^{(k)})$
\ENDFOR
\end{algorithmic}
\label{alg_global_grad}
\end{algorithm}

\begin{theorem}[Global Convergence of Algorithm~\ref{alg_global_grad}]
\label{global_newton_grad}
Let Assumptions~\eqref{strong_convex_boundedness} hold. For any $\xx^{(k)} \in  \mathbb{R}^{p}$, using Algorithm~\ref{alg_global_grad} with $\epsilon_{1} \leq 1/2$ and
\begin{equation*}
\sigma \geq \frac{4  \tilde{\kappa} }{(1 -\beta ) },
\end{equation*}
we have the following with probability $(1-\delta)^{2}$:
\begin{enumerate}[(i)]
	\item if ``STOP'', then
	\begin{equation*}
\|\nabla F(\xx^{(k)})\| < \left(1 + \sigma \right)\epsilon_{2},
\end{equation*}
\item otherwise,~\eqref{global_linear_rate} holds with $$\rho = \frac{8 \alpha_{k} \beta }{ 9 \tilde{\kappa}},$$
and the step size of at least
\begin{equation*}
\alpha_{k} \geq \frac{(1- \beta)(1-\epsilon_{1})}{\kappa},
\end{equation*}
where $\kappa$ and $\tilde{\kappa}$ are defined in~\eqref{cond_F} and~\eqref{cond_w_wo_rep}, respectively.
\end{enumerate}
\end{theorem}

Theorem~\ref{global_newton_grad} guarantees global convergence with at least a linear rate which depends on the quantities related to the specific problem, i.e., condition number. As mentioned before, in SSN2~\cite{romassn2}, we have shown, through a finer grained analysis, that the locally linear convergence rate of such sub-sampled Newton method with a constant step size $\alpha_{k} = 1$ is indeed problem independent. As in Theorem~\ref{local_global_newton}, it is possible to combine the two results and obtain a globally convergent algorithm with fast and \textit{problem-independent} rate.

\begin{theorem}[Global Conv.\ of Alg.\ \ref{alg_global_grad} with Problem-Independent Local Rate]
\label{local_global_newton_grad}
Let Assumptions~\eqref{strong_convex_boundedness} and~\eqref{Hess_Lip} hold. Consider any $0 < \rho_{0}, \rho_{1}, \rho_{2} < 1$ such that $\rho_{0} + \rho_{1} < \rho_{2}$, set $$\epsilon_{1} \leq \min\left\{\frac{(1-2\beta)}{2(1-\beta)},\frac{\rho_{0}}{4(1 + \rho_{0})\sqrt{\kappa_{1}}}\right\},$$ and define
\begin{equation*}
 c \defeq \frac{2 ( \rho_{2} - (\rho_{0} + \rho_{1}) ) (1-\epsilon_{1}) \gamma }{L}.
\end{equation*}	
Using Algorithm~\ref{alg_global_grad} with any $\xx^{(0)} \in  \mathbb{R}^{p}$ and 
\begin{align*}
&\widehat{\alpha} = 1, \quad \beta \leq \hf, \quad \sigma \geq \frac{4  \tilde{\kappa} }{(1 -\beta ) }, \\
&\epsilon_{2}^{(0)} \leq \frac{(1-\epsilon_{1})\gamma \rho_{1} (1 -2\epsilon_{1} - 2 (1-\epsilon_{1}) \beta)^{2} c}{6L \sqrt{\tilde{\kappa}}},\\
& \epsilon^{(k)}_{2} = \rho_{2} \epsilon_{2}^{(k-1)}, \quad k = 1,2,\cdots,
\end{align*}
after
\begin{equation}
k \geq \ln \left( \frac{2( \rho_{2} - (\rho_{0} + \rho_{1}) )^{2} q_{2}^{2}(\epsilon_{1},\epsilon_{2},\beta,\tilde{\kappa},L) }{9 K\left(F(\xx^{(0)}) - F(\xx^{*})\right)} \right) / \ln\left(1-\frac{8\beta(1-\beta)(1-\epsilon_{1})}{9\kappa \tilde{\kappa}}\right)
\label{local_k_grad}
\end{equation}
iterations, we have the following with probability $(1-\delta)^{2k}$:
\begin{enumerate}
	\item if ``STOP'', then
\begin{equation*}
\|\nabla F(\xx^{(k)})\| < \left(1 + \sigma \right)\rho_{2}^{k} \epsilon_{2}^{(0)},
\end{equation*}
\item otherwise, we get ``problem-independent'' linear convergence, i.e.,
\begin{equation}
\|\xx^{(k+1)} - \xx^{*}\| \leq \rho_{2} \|\xx^{(k)} - \xx^{*}\|,
\label{local_rate_grad}
\end{equation}
where $\kappa$, $\kappa_{1}$, $\tilde{\kappa}$ and $q_{2}(\epsilon_{1},\epsilon_{2},\beta,\tilde{\kappa},L)$ are defined in~\eqref{cond_F},~\eqref{cond_q},~\eqref{cond_w_wo_rep} and~\eqref{q2}, respectively.
Moreover, the step size of $\alpha_{k} = 1$ is selected in~\eqref{global_alpha_grad} for all subsequent iterations.
\end{enumerate}
\end{theorem}

\subsubsection{Inexact update}
\label{sec:global_con_nm_inexact_grad}
As in Section~\ref{sec:global_con_nm_inexact}, we now consider the inexact version of~\eqref{global_p_grad}, as a solution of 
\begin{subequations}
\begin{eqnarray}
&&\|H(\xx^{(k)})\pp_{k} + \bgg(\xx^{(k)})\| \leq \theta_{1} \|\bgg(\xx^{(k)})\|, \label{global_grad_p_inexact} \\
&&\pp_{k}^{T} \bgg(\xx^{(k)}) \leq -(1-\theta_{2}) \pp_{k}^{T} H(\xx^{(k)})\pp_{k}, \label{angle_grad_inexact}
\end{eqnarray}
\label{global_grad_update_inexact}
\end{subequations}
for some $0 \leq \theta_{1},\theta_{2} < 1$.

\begin{theorem}[Global Convergence of Algorithm~\ref{alg_global_grad}: Inexact Update]
\label{global_newton_grad_inexact}
Let Assumptions~\eqref{strong_convex_boundedness} hold. Also let $0 < \theta_{2} < 1$ and $0 < \theta_{2} < 1$ be given. For any $\xx^{(k)} \in  \mathbb{R}^{p}$, using Algorithm~\ref{alg_global_grad} with $\epsilon_{1} \leq 1/2$, the ``inexact'' update direction~\eqref{global_grad_update_inexact} instead of~\eqref{global_p_grad}, and 
\begin{equation*}
\sigma \geq  \frac{4  \tilde{\kappa}}{(1 -\theta_{1})(1 -\theta_{2}) (1 -\beta )},
\end{equation*}
we have the following with probability $1-\delta$:
\begin{enumerate}[(i)]
	\item if ``STOP'', then
	\begin{equation*}
\|\nabla F(\xx^{(k)})\| < \left(1 + \sigma \right)\epsilon_{2},
\end{equation*}
\item otherwise,~\eqref{global_linear_rate} holds where
\begin{enumerate}[(1)]
	\item if $$\theta_{1} \leq \sqrt{\frac{(1-\epsilon_{1})}{4 \tilde{\kappa}}},$$
	then $\rho = {4 \alpha_{k} \beta }/{ 9 \tilde{\kappa}}$,
\item otherwise $\rho = {8 \alpha_{k} \beta(1-\theta_{2})(1-\theta_{1})^{2}(1-\epsilon_{1}) }/{ 9\tilde{\kappa}^{2}}$,
\end{enumerate}
with $\tilde{\kappa}$ defined as in~\eqref{cond_w_wo_rep}.
Moreover, for both cases, the step size is at least
\begin{equation*}
\alpha_{k} \geq \frac{(1-\theta_{2})(1- \beta)(1-\epsilon_{1})}{\kappa},
\end{equation*}
where $\kappa$ is defined as in~\eqref{cond_F}.
\end{enumerate}
\end{theorem}

\comment Theorem~\ref{global_newton_grad_inexact} indicates that, in order to grantee a faster convergence rate, the linear system needs to be solved to a ``small-enough'' accuracy, which is in the order of $\mathcal{O}(\sqrt{1/\tilde{\kappa}})$. As in Theorem~\ref{global_newton_inexact}, we note that using a tolerance of order $\mathcal{O}(\sqrt{1/\tilde{\kappa}})$, we can still guarantee a similar global convergence rate as that of the algorithm with exact updates!

\section{Examples}
\label{sec:examples}
In this Section, we present an instance of problems which are of the form~\eqref{obj}. Specifically, examples from generalized linear models (GLM) are given in Sections~\ref{sec:example_glm}, followed by some numerical simulations in Section~\ref{simulations}. 
%Admittedly, examples in Sections~\ref{sec:example_glm} are not very good representatives of the class of problems where our methods are best suited for. More precisely, our convergence analysis is based on the assumption that only $F$ is strongly convex and no such assumption is needed on each individual $f_{i}$ (which, otherwise, would have been a very strong assumption). Unfortunately, in these examples, each $f_{i}$ is, in fact, strongly convex. This makes the problem of sub-sampling almost trivial as a sample size of $|\mathcal{S}| = 1$ would suffice to guarantee descent direction. However, our analysis applies to a much larger class of problems than the following examples imply. Despite this misfortune, these examples should be viewed as illustrations of how to estimate the different constants appearing in this paper, such as $K_{i}$, $\gamma$ and $G(\xx)$ and when to consider modifying the sub-sampled Hessian as in Section~\ref{sec:modified_hessian}.  
 
\subsection{Parameter Estimation with GLMs}
\label{sec:example_glm}
The class of generalized linear models is used to model a wide variety of regression and classification problems. The process of data fitting using such GLMs usually consists of a training data set containing $n$ response-covariate pairs, and the goal is to predict some output response based on some covariate vector, which is given after the training phase. More specifically, let $(\aa_{i},b_{i}), \; i = 1,2,\cdots,n,$ form such response-covariate  pairs in the training set where $\aa_{i} \in \mathbb{R}^{p}$. The domain of $b_{i}$ depends on the GLM used: for example, in the standard linear Gaussian model $b_{i} \in \mathbb{R}$, in the logistic models for classification, $b_{i} \in \{0, 1\}$, and in Poisson models for count-valued responses, $b_{i} \in \{0, 1, 2, \ldots \}$. See the book~\cite{mccullagh1989generalized} for further details and applications.

Consider the problem of maximum a posteriori (MAP) estimation using any GLM with canonical link function and Gaussian prior. This problem boils down to minimizing the regularized negative log-likelihood as
\begin{equation*}
F(\xx) = \frac{1}{n} \sum_{i=1}^{n} \left( \Phi(\aa_{i}^{T} \xx) - b_{i} \aa_{i}^{T} \xx \right) + \frac{\lambda}{2} \|\xx\|^{2},
\end{equation*}
where $\lambda \geq 0$ is the regularization parameter. The \textit{cumulant generating function}, $\Phi$, determines the type of GLM. For example, $\Phi(t) = 0.5 t^{2}$ gives rise to ridge regression (RR), while $\Phi(t) = \ln\left(1+\exp(t)\right)$ and $\Phi(t) = \exp(t)$ yield $\ell_{2}$-regularized logistic regression (LR) and $\ell_{2}$-regularized Poisson regression (PR), respectively. 
It is easily verified that the gradient and the Hessian of $F$ are
\begin{align*}
\nabla F(\xx) &=  \frac{1}{n} \sum_{i=1}^{n} \left( \frac{d \Phi(t)}{d t} \rvert_{t = \aa_{i}^{T} \xx} - b_{i} \right) \aa_{i} + \lambda \xx, \\
\nabla^{2} F(\xx) &=  \frac{1}{n} \sum_{i=1}^{n} \left(  \frac{d^{2} \Phi(t)}{d t^{2}} \rvert_{t = \aa_{i}^{T} \xx} \right) \aa_{i} \aa^{T}_{i} + \lambda \mathbb{I}.
\end{align*}

As mentioned before in Section~\ref{sec:subsampl_newton_hess_grad}, in order to use Lemma~\ref{randnla_lemma}, we need to be able to efficiently estimate $G(\xx^{(k)})$ at every iteration. For illustration purposes only, Table~\ref{table_GLM} gives some very rough estimates of  $G(\xx)$ for GLMs. In practice, as a pre-processing and before starting the algorithms, one can pre-compute the quantities which depend only on $(\aa_{i}, b_{i})$'s. Then updating $G(\xx)$ at every iteration is done very efficiently as it is just a matter of computing $\|\xx\|$. %In fact, for PR, the factor $K$ also depends on the current iterate. This indeed does not cause any issue as one can see from the proof of Lemma~\ref{chernoff_lemma}, that $K$ can readily be replaced by an iteration-dependent estimate.

\begin{table}[htb]
\centering
\scalebox{1}{
 \begin{tabular}{||c || c || c ||} 
 %\hline
 %\multicolumn{4}{|c|}{$G$ for GLMs with sparsity constraint} \\
 \hline
  & $\nabla f_{i}(\xx)$ & $G(\xx)$ \\ [0.5ex] 
 \hline\hline
 RR & $\left( \aa_{i}^{T} \xx - b_{i} \right) \aa_{i} + \lambda \xx$ & $\|\xx\| \max_{i} (\|\aa_{i}^{T}\|^{2} + \lambda ) + \max_{i} |b_{i}| \|\aa_{i}\|$  \\ [2ex]
 LR & $\left( \frac{1}{1+e^{-\aa_{i}^{T} \xx}}  - b_{i} \right) \aa_{i} + \lambda \xx$  & $\lambda \|\xx\| + \max_{i} (1 + |b_{i}|) \|\aa_{i}\|$ \\ [2ex]
 PR & $\left( e^{\aa_{i}^{T} \xx} - b_{i} \right) \aa_{i} + \lambda \xx$ & $\lambda \|\xx\| + e^{\hf \|\xx\|^{2}} \max_{i} \|\aa_{i}\| e^{\hf \|\aa_{i}\|^{2}} + \max_{i} |b_{i}| \|\aa_{i}\|$  \\ [2ex]
 \hline
 \end{tabular}}
\caption{Estimates for $G(\xx)$ in GLMs
\label{table_GLM}}
\end{table}

\subsection{Numerical Simulations}
\label{simulations}
In this section, we study the performance of Algorithm~\ref{alg_global} (henceforth called SSN), both with exact and inexact updates, through simulations. We consider $\ell_{2}$-regularized logistic (LR) regression as described in Section~\ref{sec:example_glm}. We use three synthetic data matrix as described in Table~\ref{synth_data} and %The condition numbers of the problems are estimated by probing the space around the optimum through multiple simulations and taking $\max_{\xx} \kappa \left( \nabla^{2} F(\xx)\right)$. 
compare the performance of the following algorithms: 
\begin{compactenum}[(i)]
	\item Gradient Descent (GD) with constant step-size (the step-size was hand tuned to obtain the best performance),
	\item Accelerated Gradient Descent (AGD),~\cite{nesterov2004introductory}, which improves over GD by using a momentum term,
	\item BFGS with Armijo line-search,
	\item L-BFGS with Armijo line-search and using limited past memory of $10$, $100$ 
	\item Full Newton's method with Armijo line-search,
	\item SSN with exact update (SSN-X) and 
	\item SSN with inexact update (SSN-NX)  with inexactness tolerances of $(\theta_{1} = 10^{-2}, \theta_{2} = 0.5)$ for data sets $D_{1}$ and $D_{2}$, and $(\theta_{1} = 10^{-4}, \theta_{2} = 0.5)$ for $D_{3}$.
\end{compactenum}

\begin{table}[htb]
\vskip 0.15in
\begin{center}
%\scalebox{0.95}{
%\begin{small}
\begin{sc}
\begin{tabular}{|cccccc|}
\hline
Data & $n$ & $p$ & nnz & $\kappa $ & $\kappa_{1}$ \\
\hline
$D_{1}$  &  $10^{6}$ &  $10^{4}$ & $0.02\%$ & $\approx 10^{4}$ & $\approx 10^{6}$\\
$D_{2}$  &  $5\times10^{4}$ &  $5\times10^{3}$ & Dense & $\approx 10^{6}$ & $\approx 10^{6}$\\
$D_{3}$  &  $10^{7}$ &  $2\times10^{4}$ &  $0.006\%$ & $\approx 10^{10}$ & $\approx 10^{11}$\\
\hline
\end{tabular}
\end{sc}
%\end{small}
%}
\end{center}
\vskip -0.1in
\caption{Synthetic Data sets used in the experiments. ``nnz'' refers to the number of non-zeros in the data set. $\kappa$ and $\kappa_{1}$ are the condition number of $F$ and that of the sub-sampling problem, defined in~\eqref{cond_F} and~\eqref{cond_q}, respectively.} \label{synth_data}
\end{table}

We run the simulations for each method, starting from the same initial point, until $\|\nabla F(\xx)\| \leq 10^{-8}$ or a maximum number of iterations is reached, and report the relative errors of the iterates, i.e., $\|\xx^{(k)}- \xx^{*} \|/\|\xx^{*}\|$ as well as the relative errors of the objective function, i.e., $|F(\xx^{(k)})- F(\xx^{*}) |/|F(\xx^{*})|$, both versus elapsed time (in seconds). %For data sets $D_{1}$ and $D_{2}$, we set the regularization $\lambda = 10^{-6}$, while for $A_{3}$, it is set at $\lambda = 10^{-9}$. 
The results are shown in Figure~\ref{fig_results}. 

The first order methods, i.e., GD and AGD, in none of these examples, managed to converge to anything reasonable. For the ill-conditioned problems with data sets $D_{2}$ and $D_{3}$, while all instances of SSN converged to the desired accuracy with no difficulty, no other method managed to go past a ``single'' digit of accuracy, in the same time-period, or anytime soon after! This is indeed expected, as SSN captures the regions with \textit{high} and \textit{low} curvature, and scales the gradient accordingly to make progress. These examples show that, when dealing with ill-conditioned problems, using only first order information is certainly not enough to obtain a reasonable solution. In these problems, employing a second-order algorithm such as SSN with inexact update not only yields the desired solution, but also does it very efficiently! In particular note the fast convergence of SSN-NX for both of these problems. 

For the much better conditioned problem using the data set $D_{1}$, BFGS and L-BFGS with the history size of $100$ outperform SSN-X with $5\%$ and $10\%$ sampling. This is also expected since solving the $10,000 \times 10,000$ linear system exactly at every iteration is the bottleneck of computations for SSN-X, in comparison to matrix free BFGS and L-BFGS. However, even in such a well-conditioned problem where BFGS and L-BFGS appear very attractive, inexactness coupled with SSN can be more efficient. It is clear that all SSN-NX variants converge to the desired solution faster than either BFGS or its limited memory variant. For example, using $D_{1}$, the speed-up of using SSN-NX with $10\%$ of the data over L-BFGS is at least $4$ times.  Note also that for the very ill-conditioned problem with $D_{3}$, larger sample size, i.e., $20\%$, was required to obtain a fast convergence, illustrating that the sample size could grow with the condition number.

\begin{figure}[!ht]\centering
    \begin{minipage}[htb]{0.32\linewidth}
        \subfigure[Iterate Rel.\ Err.\ for $D_{1}$]
        {\includegraphics[scale=0.29]{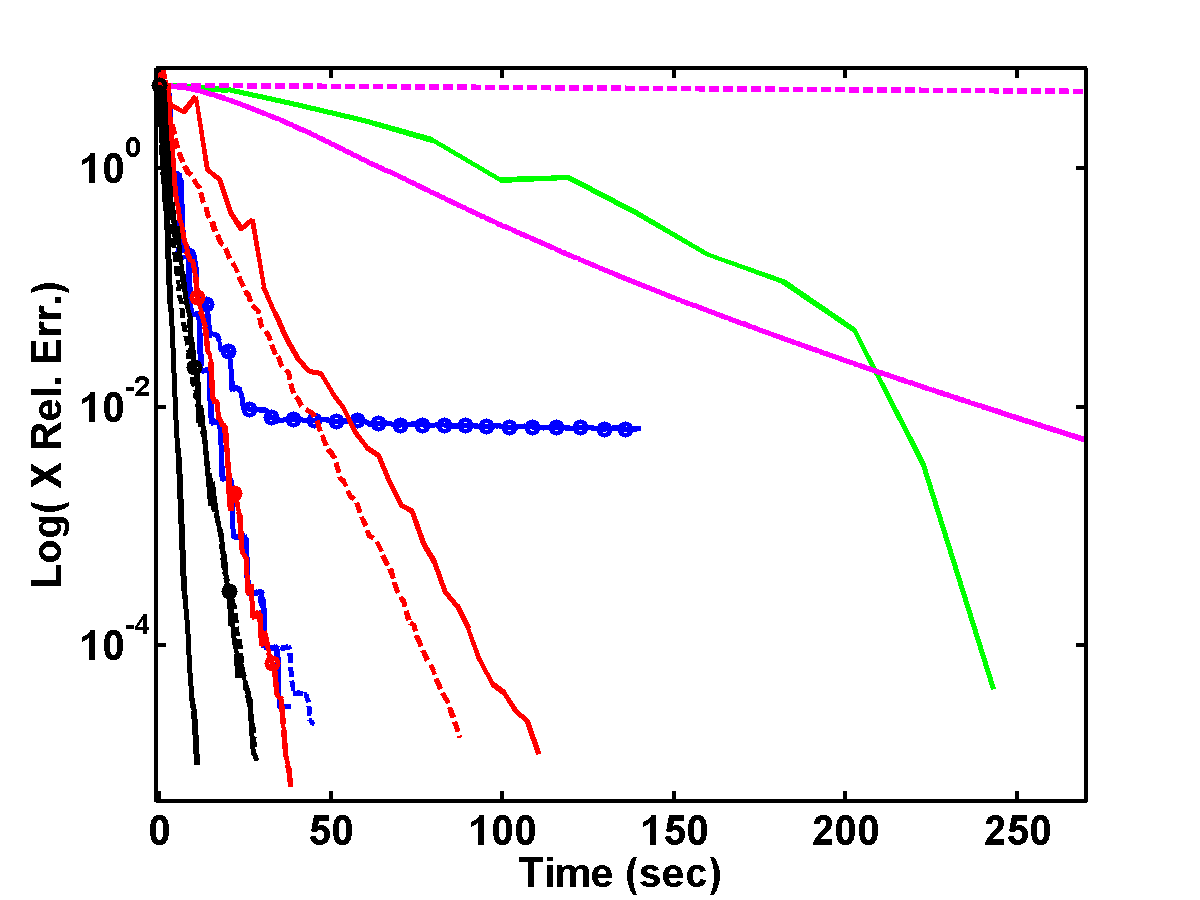}}
    \end{minipage}
    \begin{minipage}[htb]{0.32\linewidth}
        \subfigure[Function Rel.\ Err.\ for $D_{1}$]
        {\includegraphics[scale=0.29]{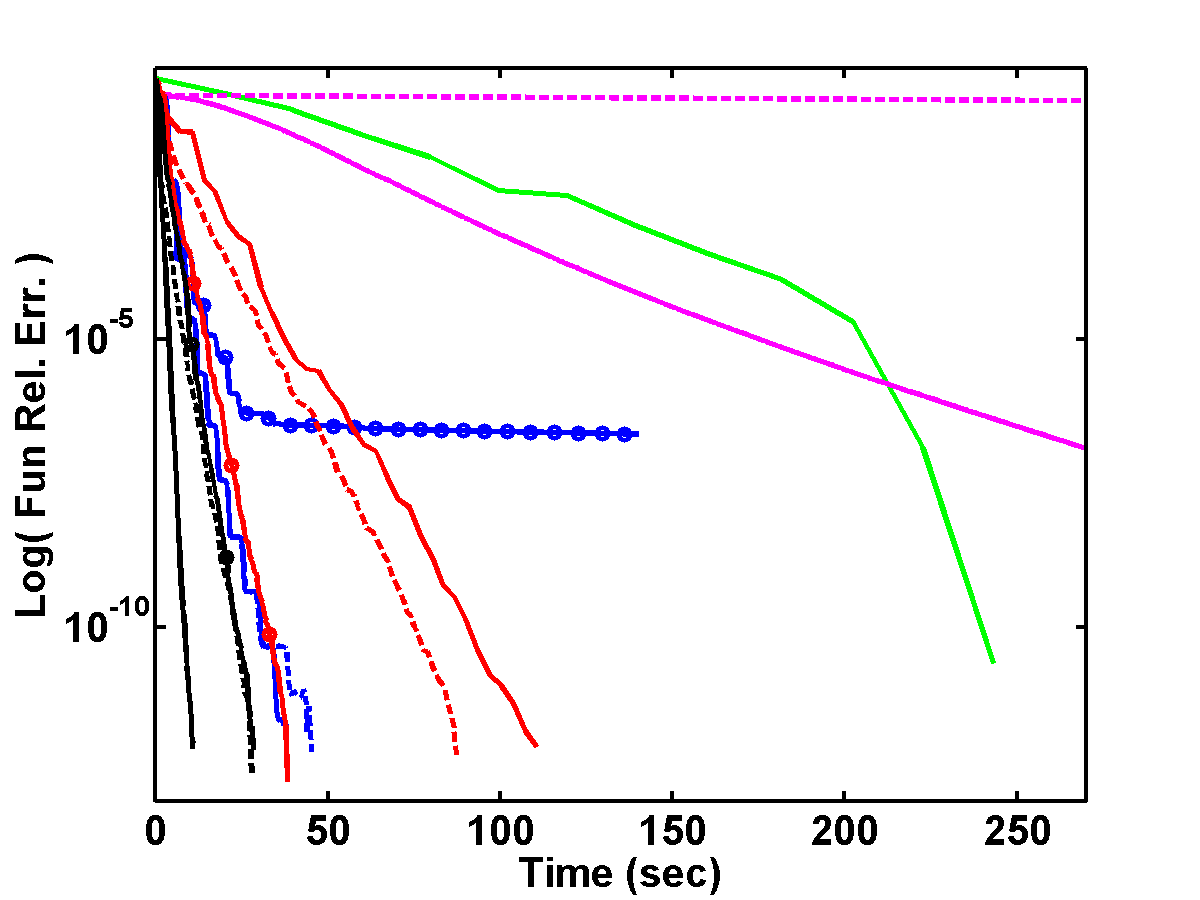}}
    \end{minipage}
		\begin{minipage}[htb]{0.32\linewidth}
        \subfigure[Iterate Rel.\ Err.\ for $D_{2}$]
        {\includegraphics[scale=0.29]{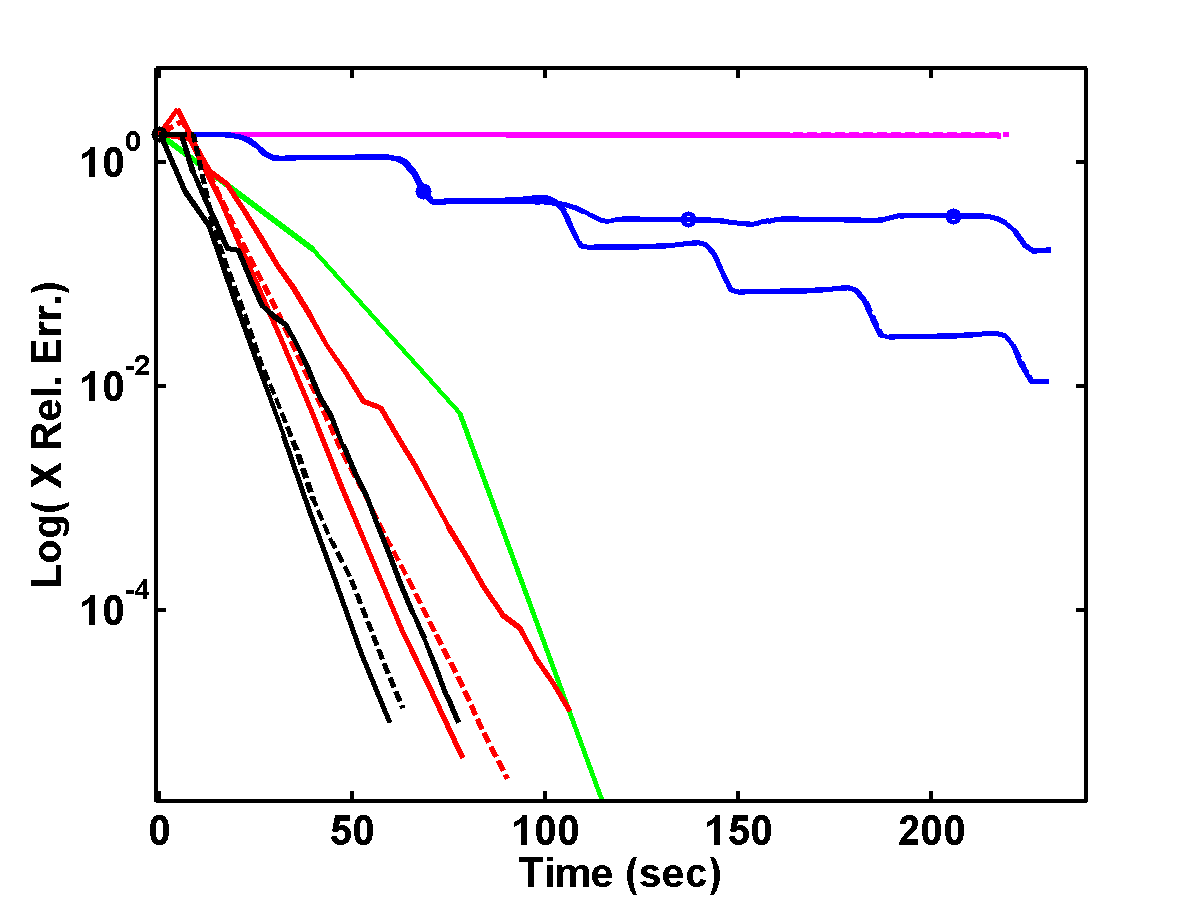}}
    \end{minipage}
    \begin{minipage}[htb]{0.32\linewidth}
        \subfigure[Function Rel.\ Err.\ for $D_{2}$]
        {\includegraphics[scale=0.29]{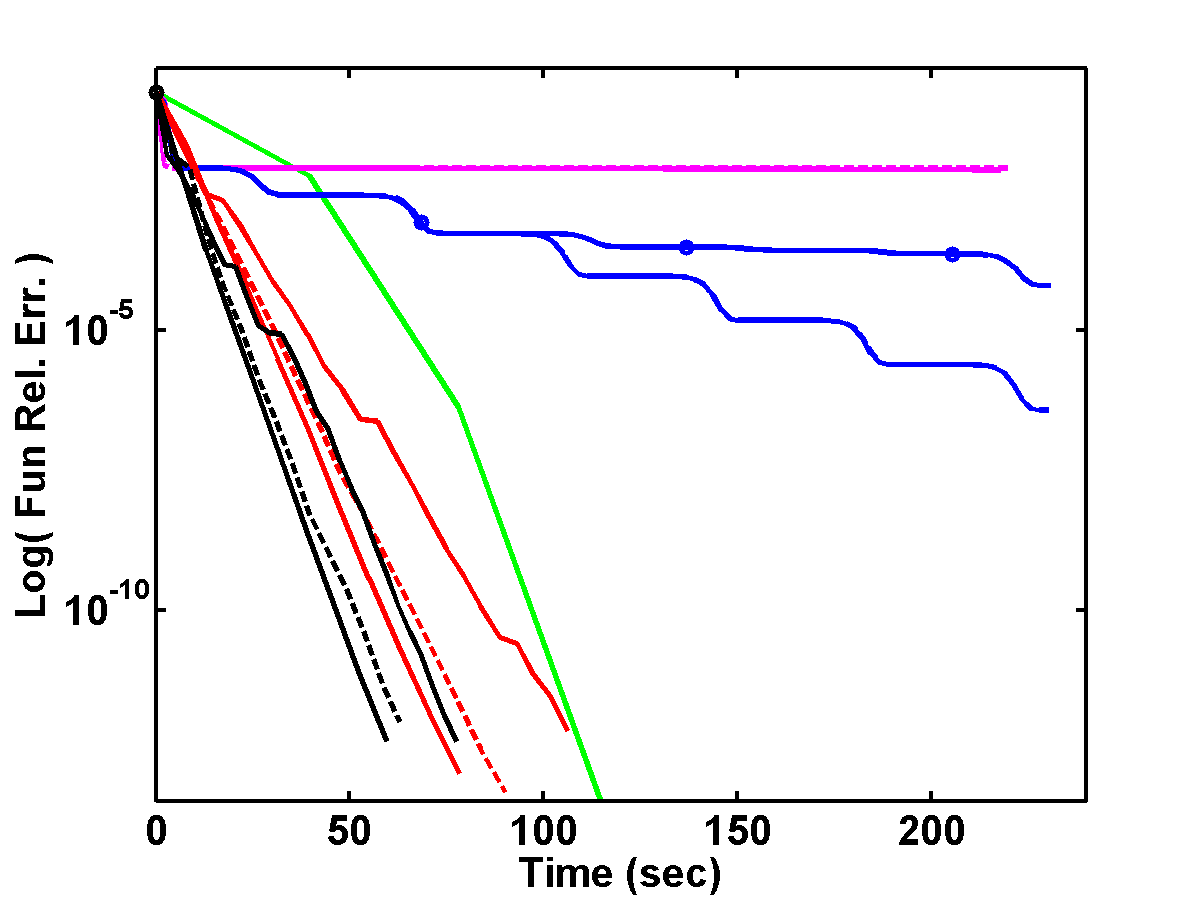}}
    \end{minipage}
		\begin{minipage}[htb]{0.32\linewidth}
        \subfigure[Iterate Rel.\ Err.\ for $D_{3}$]
        {\includegraphics[scale=0.29]{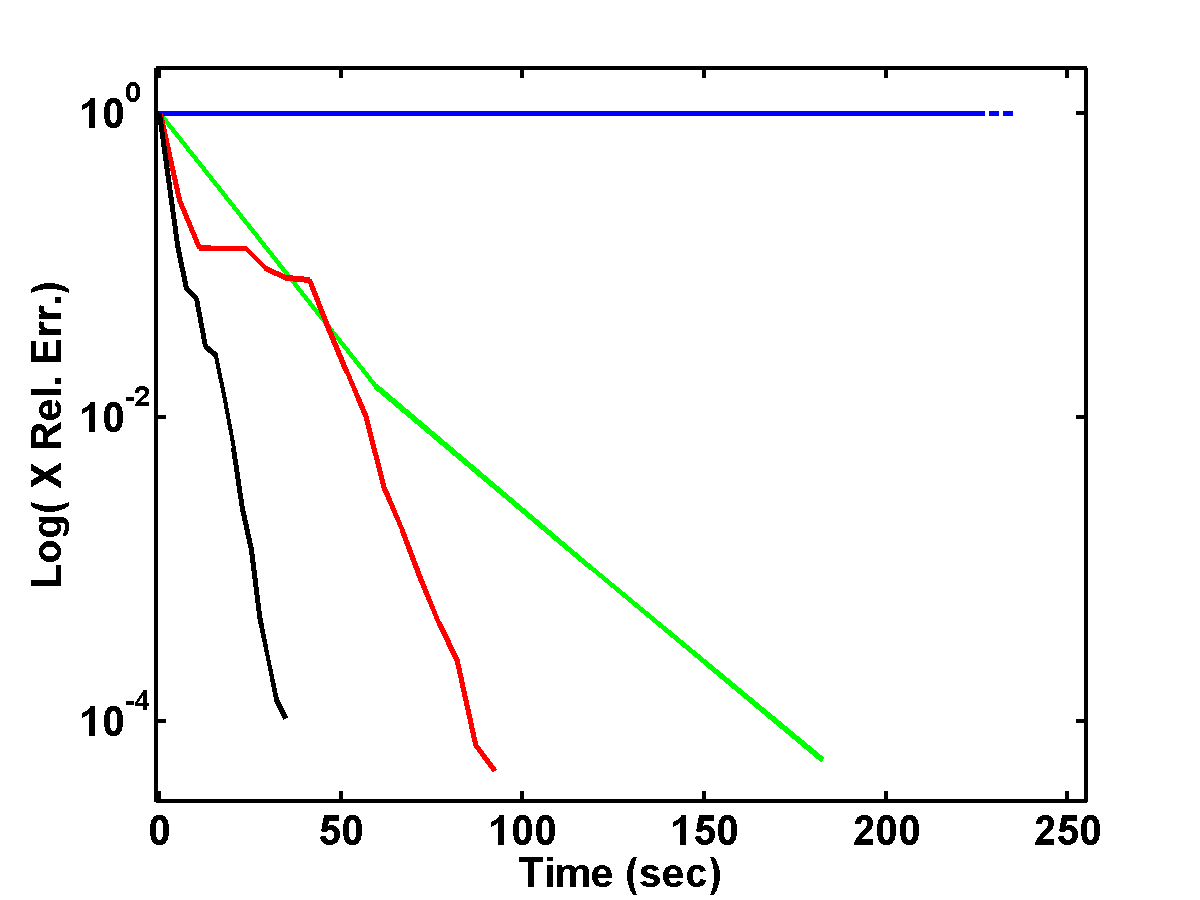}}
    \end{minipage}
    \begin{minipage}[htb]{0.32\linewidth}
        \subfigure[Function Rel.\ Err.\ for $D_{3}$]
        {\includegraphics[scale=0.29]{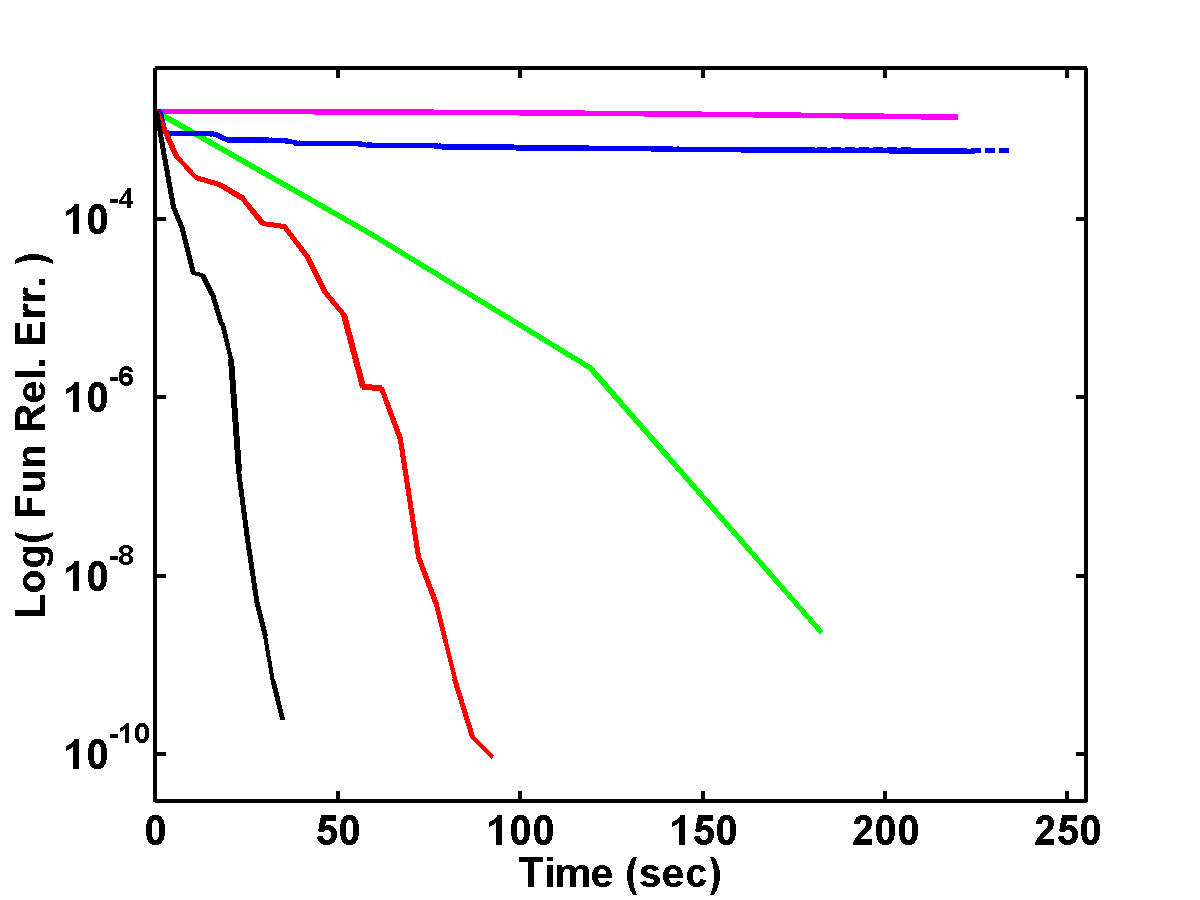}}
    \end{minipage}
		\begin{minipage}[htb]{0.2\linewidth}
        \subfigure[Legend for Figures (a)-(d)]
        {\includegraphics[scale=0.39]{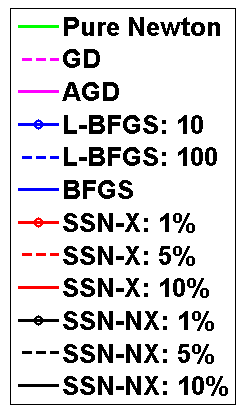}}
    \end{minipage}
		\begin{minipage}[htb]{0.2\linewidth}
        \subfigure[Legend for Figures (e)-(f)]
        {\includegraphics[scale=0.39]{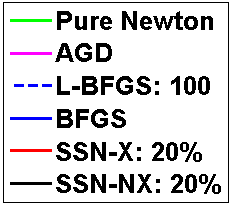}}
    \end{minipage}
    \caption{Comparison of different methods w.r.t.\ iterate and function value relative errors for $\ell_{2}$-regularized LR using data sets $D_{1}$, $D_{2}$ and $D_{3}$ as described in Table~\ref{synth_data}. The percentage values, e.g., SSN-X: $10\%$, refer to sub-sampling ratio, i.e., $|\mathcal{S}|/n$.}
\label{fig_results}
\end{figure}

\section{Conclusion}
\label{sec:conclusion}
In this paper, we studied globally convergent sub-sampled Newton algorithms for unconstrained optimization in various settings. In the first part of the paper, we studied the case in which only the Hessian is sub-sampled and the full gradient is used. In this setting, we showed that using random matrix concentration result, it is possible to, probabilistically, guarantee that the sub-sampled Hessian yields a descent direction at every iteration. We then provided global convergence for modifications of this algorithms where the sub-sampled Hessian is regularized by changing its spectrum or ridge-type regularization.  We argued that such regularization can only be beneﬁcial at early stages of the algorithm and we need to revert to using the true sub-sampled Hessian (of course, if it is invertible) as iterates get closer to the optimum. 

In the second part of the paper, we considered the global convergence of a fully stochastic algorithm in which both the Hessian and the gradients are sub-sampled, independently of each other, as way to further reduce the computational complexity per iteration. We use approximate matrix multiplication results from RandNLA to obtain the proper sampling strategy. 

In all of these algorithms, computing the update boils down to solving a large scale linear system which can be the bottleneck of computations. As a result, for all of our algorithms, in addition to giving global convergence results for the case where such linear system is solved exactly, we give similar results for the case of inexact update, where the arsing linear system is solved only approximately. In addition, we gave sufficient conditions on the accuracy tolerance to guarantee faster convergence results. In fact, we showed that the accuracy tolerance needs only to be in the order of $\mathcal{O}(\sqrt{1/\tilde{\kappa}})$, to guarantee such faster rate, where $\tilde{\kappa}$ is the sampling condition number of the problem.

Although our main focus here was merely to provide global convergence guarantees for sub-sampled Newton methods under a variety of settings, admittedly, one major downside of the bounds presented in this paper is that they are all pessimistic. In fact, some of the bounds of the present paper exhibit a dependence on the condition-number which is very discouraging. However, the consolation lies in combining the global  results presented here and the corresponding local convergence rates of the companion paper, SSN2~\cite{romassn2}. Indeed, in SSN2~\cite{romassn2}, we show that, using Newton's ``natural'' step-size of $\alpha_{k}=1$, such sub-sampled Newton algorithms enjoy local convergence rates which are \textit{condition-number independent}. In addition, we show that, through controlling the sub-sampling accuracy, one can make the local convergence speed of such algorithms as close to that of the full Newton's method as desired (though we stay shy of obtaining its famous quadratic rate). Consequently through such combination of the results of the two companion papers, we guaranteed that, ultimately, the convergence rate of the algorithms of the present paper which use exact update, becomes condition-number independent. In addition, using the Armijo rule, the ``natural'' step size of $\alpha_{k} = 1$ will eventually be  always accepted.

Finally, despite the fact that we considered the global convergence behavior of the algorithms which incorporate inexact updates, the local convergence properties of such algorithms are not known. The results of SSN2~\cite{romassn2} only address such properties for the algorithms which use exact update. As a result, studying local convergence behavior of such inexact algorithms is left for future work. In addition, here, the global convergence have been established for algorithms for solving unconstrained optimization. SSN2~\cite{romassn2} considers the general case of constrained optimization, but only in studying the local convergence rates of the presented algorithms. As a result, extensions of global convergence guarantees  to convex constrained problems are important avenues for future research.

%Finally, we note that while, here, we only focus of the global nature of the algorithms, the local convergence behavior of such sub-sampled methods is studied in the companion paper, SSN2~\cite{romassn2}.

%%%%%%%%%%%%%%%%%%%%%%%%%%%%%%%%%%%%%%%%%%%%%%%%%%%%%%%%%%%%%%%%%%%%%%%%%%%%%
%
\bibliographystyle{plain}
\bibliography{../../../biblio}

\appendix
\section{Proofs}
 
\subsection{Proofs of Section~\ref{sec:subsampl_newton_hess}}
\begin{proof}[Proof of Lemma~\ref{chernoff_lemma}]
Consider $|\mathcal{S}|$ i.i.d random matrices $X_{j}(\xx), j=1,2,\ldots,|\mathcal{S}|$ such that $\Pr(X_{j}(\xx) = \nabla^{2} f_{i}(\xx)) = 1/n; \; \forall i = 1,2,\ldots,n,$. Define
\begin{eqnarray*}
H(\xx) &\defeq& \frac{1}{|\mathcal{S}|}\sum_{j \in \mathcal{S}} X_{j}(\xx) , \\
X &\defeq& \sum_{j \in \mathcal{S}} X_{j} = |\mathcal{S}| H(\xx).
\end{eqnarray*} 
Note that $X_{j} \succeq 0$, $\Ex(X_{j}) = \nabla^{2} F(\xx) $ and 
\begin{eqnarray*}
\lambda_{\max} \Big( X_{j}\Big) &\leq& \widehat{K}_{1}, \\
\lambda_{\min} \Big(\sum_{j \in \mathcal{S}} \Ex(X_{j})\Big) &=& |\mathcal{S}| \lambda_{\min} \left( \nabla^{2} F(\xx) \right) \geq |\mathcal{S}| \gamma,
\end{eqnarray*}
where $\widehat{K}_{1}$ is defined in~\eqref{K_q}.
Hence we can apply Matrix Chernoff~\cite[Theorem 1.1]{tropp2012user} or~\cite[Theorem 2.2]{tropp2011improved} for sub-sampling with or without replacement, respectively, to get
\begin{eqnarray*}
\Pr\Big(\lambda_{\min}(X) \leq (1-\epsilon) |\mathcal{S}| \lambda_{\min} \left(  \nabla^{2} F(\xx) \right)\Big) 
\leq p \left[ \frac{e^{-\epsilon}}{(1-\epsilon)^{(1-\epsilon)}}\right]^{\mathcal{|S|}\gamma/\widehat{K}_{1}}.
\end{eqnarray*}
Now the result follows by noting that
\begin{equation*}
\frac{e^{-\epsilon}}{(1-\epsilon)^{(1-\epsilon)}} \leq e^{-\epsilon^{2}/2},
\end{equation*}
and requiring that
\begin{equation*}
e^{-\epsilon^{2} \mathcal{|S|}/(2\kappa_{1})} \leq \delta.
\end{equation*}
\qed
\end{proof}

\begin{proof}[Proof of Theorem~\ref{global_newton}:]
First note that by~\eqref{eig_chernoff}, we have
\begin{equation*}
\pp_{k}^{T} H(\xx^{(k)}) \pp_{k} \geq (1-\epsilon) \gamma \|\pp_{k}\|^{2},
\end{equation*}
which from
\begin{equation*}
\pp_{k}^{T} \nabla F(\xx^{(k)}) =  - \pp_{k}^{T} H(\xx^{(k)}) \pp_{k},
\end{equation*}
implies that $\pp_{k}^{T} \bgg(\xx^{(k)}) < 0$ and we can indeed obtain decrease in the objective function.

Now, it suffices to show that there exists an \textit{iteration-independent} $\widetilde{\alpha} > 0$, such that the constrain in~\eqref{global_alpha} holds for any $0 \leq \alpha \leq \widetilde{\alpha}$. For any $0\leq \alpha < 1$, define $\xx_{\alpha} = \xx^{(k)} + \alpha \pp_{k}$. By Assumption~\eqref{F_strong_gen}, we have

\begin{eqnarray*}
F(\xx_{\alpha}) - F(\xx^{(k)}) &\leq& (\xx_{\alpha} - \xx^{(k)})^{T}\nabla F(\xx^{(k)}) + \frac{K}{2} \|\xx_{\alpha} - \xx^{(k)}\|^{2} \\
&=& \alpha \pp_{k}^{T}\nabla F(\xx^{(k)}) + \alpha^{2} \frac{K}{2} \|\pp_{k}\|^{2}
\end{eqnarray*}
Now in order to pass the Armijo rule, we search for $\alpha$ such that
\begin{equation*}
\alpha \pp_{k}^{T}\nabla F(\xx^{(k)}) + \alpha^{2} \frac{K}{2} \|\pp_{k}\|^{2} \leq \alpha \beta \pp_{k}^{T}\nabla F(\xx^{(k)}),
\end{equation*}
which, in turn, gives
\begin{equation*}
\alpha \frac{K}{2} \|\pp_{k}\|^{2} \leq -(1-\beta) \pp_{k}^{T}\nabla F(\xx^{(k)}).
\end{equation*}
This latter inequality is satisfied if we require that
\begin{equation*}
\alpha \frac{K}{2} \|\pp_{k}\|^{2} \leq (1-\beta) \pp_{k}^{T} H(\xx^{(k)}) \pp_{k}.
\end{equation*}
As a result, having
\begin{equation*}
\alpha   \leq \frac{2(1-\beta) (1-\epsilon)}{\kappa},
\end{equation*}
satisfies the Armijo rule. So in particular, we can always find an iteration independent lower bound on step size such that the constrain in~\eqref{global_alpha} holds. On the other hand, for sampling \textit{without} replacement and from $H(\xx^{(k)}) \pp_{k} = -\nabla F(\xx^{(k)})$ we get
\begin{equation*}
\pp_{k}^{T} H(\xx^{(k)}) \pp_{k} = \nabla F(\xx^{(k)})^{T} [H(\xx^{(k)})]^{-1} \nabla F(\xx^{(k)}) \geq \frac{1}{\widehat{K}_{|\mathcal{S}|}} \| \nabla F(\xx^{(k)})\|^{2}.
\end{equation*} 
Similarly for sampling \textit{with} replacement, we have $$\pp_{k}^{T} H(\xx^{(k)}) \pp_{k} \geq {1}/{\widehat{K}_{1}} \| \nabla F(\xx^{(k)})\|^{2}.$$ Now the result follows immediately by noting that Assumption~\eqref{F_strong_gen} implies (see~\cite[Theorem 2.1.10]{nesterov2004introductory})
\begin{equation*}
F(\xx^{(k)}) - F(\xx^{*}) \leq \frac{1}{2 \gamma} \|\nabla F(\xx^{(k)})\|^{2}.
\end{equation*}
\qed
\end{proof}

\begin{proof}[Proof of Theorem~\ref{local_global_newton}]
The choice of $\epsilon$ is to meet a requirement of SSN2~\cite[Theorem 2]{romassn2} and account for the differences between Lemma~\ref{chernoff_lemma} and SSN2~\cite[Lemma 1]{romassn2}. 

The rest of the proof follows closely the line of argument in~\cite[Section 9.5.3]{boyd2004convex}. define $\xx_{\alpha} = \xx^{(k)} + \alpha \pp_{k}$. From~\eqref{Hess_Lip}, it follows that
\begin{equation*}
\|\nabla^{2}F(\xx_{\alpha}) - \nabla^{2}F(\xx^{(k)})\| \leq L \|\xx_{\alpha} - \xx^{(k)}\|,
\end{equation*} 
which implies that
\begin{equation*}
\pp_{k}^{T}\left(\nabla^{2}F(\xx_{\alpha}) - \nabla^{2}F(\xx^{(k)})\right) \pp_{k} \leq \alpha L \|\pp_{k}\|^{3},
\end{equation*} 
which, in turn, gives
\begin{equation*}
\pp_{k}^{T} \nabla^{2}F(\xx_{\alpha}) \pp_{k} \leq \pp_{k}^{T} \nabla^{2}F(\xx^{(k)}) \pp_{k} + \alpha L \|\pp_{k}\|^{3}.
\end{equation*} 
Defining $\widehat{F}(\alpha) \defeq F(\xx^{(k)} + \alpha \pp_{k})$, we have
\begin{equation*}
\widehat{F}^{''}(\alpha) \leq \widehat{F}^{''}(0) + \alpha L \|\pp_{k}\|^{3}.
\end{equation*}
Now we integrate this inequality to get
\begin{equation*}
\widehat{F}^{'}(\alpha) \leq \widehat{F}^{'}(0) + \alpha \widehat{F}^{''}(0) + \frac{\alpha^{2}}{2} L \|\pp_{k}\|^{3}.
\end{equation*}
Integrating one more time yields
\begin{equation*}
\widehat{F}(\alpha) \leq \widehat{F}(0) + \alpha \widehat{F}^{'}(0) + \frac{\alpha^{2}}{2} \widehat{F}^{''}(0) + \frac{\alpha^{3}}{6} L \|\pp_{k}\|^{3}.
\end{equation*}
On the other hand, we have
\begin{align*}
\|\pp_{k}\|^{2} = \|[H(\xx^{(k)})]^{-1} \nabla F(\xx^{(k)})\|^{2} \leq \frac{1}{(1-\epsilon)\gamma} \nabla F(\xx^{(k)})^{T} [H(\xx^{(k)})]^{-1} \nabla F(\xx^{(k)}),
\end{align*}
as well as $\widehat{F}^{'}(0) = \alpha \pp_{k}^{T} \nabla F(\xx^{(k)}) = - \alpha \nabla F(\xx^{(k)})^{T} [H(\xx^{(k)})]^{-1} \nabla F(\xx^{(k)})$ and
\begin{align*}
\widehat{F}^{''}(0) = \alpha^{2} \pp_{k}^{T} \nabla^{2} F(\xx^{(k)}) \pp_{k}  &= \alpha^{2} \nabla F(\xx^{(k)})^{T} [H(\xx^{(k)})]^{-1} \nabla^{2} F(\xx^{(k)}) [H(\xx^{(k)})]^{-1} \nabla F(\xx^{(k)}) \\
& \leq \frac{\alpha^{2}}{(1-\epsilon)} \nabla F(\xx^{(k)})^{T} [H(\xx^{(k)})]^{-1} \nabla F(\xx^{(k)}).
\end{align*}
The last inequality follows since by the choice of $\epsilon$ and SSN2~\cite[Lemma 5]{romassn2}, we have
\begin{align*}
\|H(\xx^{(k)}) -\nabla^{2} F(\xx^{(k)}) \| \leq \epsilon \gamma,
\end{align*}
and hence, for any $\vv$
\begin{align*}
\vv^{T} [H(\xx^{(k)})]^{-1} \nabla^{2} F(\xx^{(k)}) [H(\xx^{(k)})]^{-1} \vv - \vv^{T} [H(\xx^{(k)})]^{-1} \vv &\leq \epsilon \gamma \vv^{T} [H(\xx^{(k)})]^{-2} \vv \\
&\leq \frac{\epsilon}{1-\epsilon} \vv^{T} [H(\xx^{(k)})]^{-1} \vv,
\end{align*}
which gives
\begin{align*}
\vv^{T} [H(\xx^{(k)})]^{-1} \nabla^{2} F(\xx^{(k)}) [H(\xx^{(k)})]^{-1} \vv \leq \frac{1}{(1-\epsilon)} \vv^{T} [H(\xx^{(k)})]^{-1} \vv.
\end{align*}
Hence, with $\alpha = 1$ and denoting $c(\xx) \defeq \nabla F(\xx)^{T} [H(\xx)]^{-1} \nabla F(\xx)$, we have
\begin{align*}
F(\xx^{(k)} + \pp_{k}) &\leq F(\xx^{(k)}) + \left(\frac{1}{2 (1-\epsilon)} -1 \right) c(\xx^{(k)}) + \frac{L}{6}\left( \frac{1}{(1-\epsilon)\gamma} c(\xx^{(k)})\right)^{3/2} \\
& \leq  F(\xx^{(k)}) + c(\xx) \left(\frac{1}{2 (1-\epsilon)} -1 + \frac{L}{6} \left(\frac{1}{(1-\epsilon)\gamma}\right)^{3/2} c(\xx^{(k)})^{1/2} \right).
\end{align*}
Hence, noting that $c(\xx) \leq \|\nabla F(\xx)\|^{2}/((1-\epsilon) \gamma)$, if 
%\begin{equation*}
%\left(\frac{1}{2 (1-\epsilon)} -1 + \frac{L}{6} \frac{1}{(1-\epsilon)^{2}\gamma^{2}} \|\nabla F(\xx)\| \right) \leq -\beta
%\end{equation*}
\begin{equation}
\label{grad_small}
\|\nabla F(\xx^{(k)})\| \leq  \frac{3(1-\epsilon)\gamma^{2}\big(1-2\epsilon - 2 (1-\epsilon) \beta\big)}{L},
\end{equation}
we get
\begin{align*}
F(\xx^{(k)} + \pp_{k}) \leq F(\xx^{(k)}) -\beta \nabla F(\xx)^{T} [H(\xx)]^{-1} \nabla F(\xx) = F(\xx^{(k)}) +\beta \pp_{k}^{T} \nabla F(\xx),
\end{align*}
which implies that~\eqref{global_alpha} is satisfied with $\alpha = 1$. 

The proof is complete if we can find $k$ such that both the sufficient condition of SSN2~\cite[Theorem 2]{romassn2} as well as~\eqref{grad_small} is satisfied. First, note that from Theorem~\ref{global_newton}, Assumtpion~\eqref{F_strong_gen} and by using the iteration-independent lower bound on $\alpha_{k}$, it follows that
\begin{equation*}
\|\nabla^{2}F(\xx^{(k)})\|^{2} \leq 2 K (1 - \hat{\rho})^{k} \left(F(\xx^{(0)}) - F(\xx^{*}) \right),
\end{equation*}
where
\begin{equation*}
\hat{\rho} = \frac{4 \beta  (1- \beta)(1-\epsilon) }{\tilde{\kappa} \kappa}.
\end{equation*}
In order to satisfy~\eqref{grad_small},  we require that 
\begin{equation*}
2 K (1 - \hat{\rho})^{k} \left(F(\xx^{(0)}) - F(\xx^{*}) \right) \leq \frac{4(1-\epsilon)^{2}\gamma^{4}\big(1-2\epsilon - 2 (1-\epsilon) \beta\big)^{2} (\rho_{1} - \rho_{0})^{2}}{L^{2}},
\end{equation*}
which yields~\eqref{local_k}. Again, from Theorem~\ref{global_newton} and Assumtpion~\eqref{F_strong_gen},we get
\begin{equation*}
\|\xx^{(k)} - \xx^{*}\|^{2} \leq \frac{2 (1 - \hat{\rho})^{k}}{\gamma} \left(F(\xx^{(0)}) - F(\xx^{*}) \right),
\end{equation*}
which implies that
\begin{equation*}
\|\xx^{(k)} - \xx^{*}\|^{2} \leq \frac{4(1-\epsilon)^{2}\gamma^{3}\big(1-2\epsilon - 2 (1-\epsilon) \beta\big)^{2} (\rho_{1} - \rho_{0})^{2}}{K L^{2}} \leq \frac{4(1-\epsilon)^{2}\gamma^{2}(\rho_{1} - \rho_{0})^{2}}{L^{2}},
\end{equation*}
and hence the sufficient condition of SSN2~\cite[Theorem 2]{romassn2} is also satisfied and we get~\eqref{local_rate}. \qed
\end{proof}

\begin{proof}[Proof of Theorem~\ref{global_newton_inexact}:]
We give the proof only for the case of sampling without replacement. The proof for sampling with replacement is obtained similarly. 

First, we note that~\eqref{eig_chernoff} and~\eqref{angle_inexact} imply
\begin{eqnarray}
\pp_{k}^{T} \nabla F(\xx^{(k)}) &\leq& -(1-\theta_{2}) (1-\epsilon) \gamma \|\pp_{k}\|^{2}.
\label{armijo_loose_1}
\end{eqnarray}
So, $\pp_{k}^{T} \bgg(\xx^{(k)}) < 0$ and we can indeed obtain decrease in the objective function.
As in the proof of Theorem~\eqref{global_newton}, we get
\begin{eqnarray*}
F(\xx_{\alpha}) - F(\xx^{(k)}) \leq \alpha \pp_{k}^{T}\nabla F(\xx^{(k)}) + \alpha^{2} \frac{K}{2} \|\pp_{k}\|^{2}.
\end{eqnarray*}
Hence, in order to pass the Armijo rule, we search for $\alpha$ such that
\begin{equation*}
\alpha \frac{K}{2} \|\pp_{k}\|^{2} \leq -(1-\beta) \pp_{k}^{T}\nabla F(\xx^{(k)}).
\end{equation*}
As a result, having
\begin{equation*}
\alpha   \leq \frac{2(1-\theta_{2})(1-\beta) (1-\epsilon) }{\kappa},
\end{equation*}
satisfies the Armijo rule. 

For part (i), we notice that by the self-duality of the vector $\ell_{2}$ norm, i.e.,
\begin{equation*}
\|\vv\|_{2} = \sup \{\ww^{T} \vv; \; \|\ww\|_{2}=1\},
\end{equation*}
it follows that the condition~\eqref{global_p_inexact}, implies
\begin{equation*}
\pp_{k}^{T} \nabla F(\xx^{(k)}) + \nabla F(\xx^{(k)})^{T} [H(\xx^{(k)})]^{-1} \nabla F(\xx^{(k)}) \leq \theta_{1} \|\nabla F(\xx^{(k)})\| \|[H(\xx^{(k)})]^{-1}\nabla F(\xx^{(k)})\|.
\end{equation*}
Using~\eqref{eig_chernoff}, we get
\begin{equation*}
[H(\xx^{(k)})]^{-1} \preceq \frac{1}{(1-\epsilon)\gamma},
\end{equation*}
which implies that
\begin{equation*}
\|[H(\xx^{(k)})]^{-1}\nabla F(\xx^{(k)})\| \leq \sqrt{\frac{1}{(1-\epsilon)\gamma} \nabla F(\xx^{(k)})^{T} [H(\xx^{(k)})]^{-1} \nabla F(\xx^{(k)})}.
\end{equation*}
Hence, denoting
\begin{equation*}
q \defeq \sqrt{\nabla F(\xx^{(k)})^{T} [H(\xx^{(k)})]^{-1} \nabla F(\xx^{(k)})},
\end{equation*}
we get
\begin{eqnarray*}
\pp_{k}^{T} \nabla F(\xx^{(k)}) &\leq& \frac{\theta_{1}}{\sqrt{(1-\epsilon)\gamma}} \|\nabla F(\xx^{(k)})\| q - q^{2}\\
&=& q \left( \frac{\theta_{1}}{\sqrt{(1-\epsilon)\gamma}} \|\nabla F(\xx^{(k)})\| -  q\right).
\end{eqnarray*}
Now, we require that
\begin{equation*}
\frac{\theta_{1}}{\sqrt{(1-\epsilon)\gamma}} \|\nabla F(\xx^{(k)})\| \leq \frac{q}{2},
\end{equation*}
which since
\begin{equation*}
q \geq \frac{1}{\sqrt{\widehat{K}_{|\mathcal{S}|}}} \|\nabla F(\xx^{(k)})\|,
\end{equation*}
follows if
\begin{equation*}
\theta_{1} \leq \frac{\sqrt{(1-\epsilon)\gamma}}{2\sqrt{\widehat{K}_{|\mathcal{S}|}}}.
\end{equation*}
With such $\theta_{1}$, we get
\begin{equation*}
\pp_{k}^{T} \nabla F(\xx^{(k)}) \leq - \frac{q^{2}}{2} \leq \frac{-1}{2 \widehat{K}_{|\mathcal{S}|}} \|\nabla F(\xx^{(k)})\|^{2}.
\end{equation*}

For part (ii), we have
\begin{eqnarray*}
\theta_{1} \|\nabla F(\xx^{(k)})\| &\geq& \|H(\xx^{(k)})\pp_{k} + \nabla F(\xx^{(k)})\| \\
&\geq& \|\nabla F(\xx^{(k)})\| - \|H(\xx^{(k)})\pp_{k}\|,
\end{eqnarray*}
which, in turn, implies
\begin{eqnarray*}
(1 -\theta_{1}) \|\nabla F(\xx^{(k)})\| &\leq& \|H(\xx^{(k)})\pp_{k}\| \\
&\leq& \|H(\xx^{(k)})\|\|\pp_{k}\| \\
&\leq& \widehat{K}_{|\mathcal{S}|}\|\pp_{k}\|.
\end{eqnarray*}
Hence using~\eqref{armijo_loose_1}, we get
\begin{equation*}
\pp_{k}^{T} \nabla F(\xx^{(k)}) \leq -(1-\theta_{2}) (1-\epsilon) \gamma \frac{(1 -\theta_{1})^{2}}{\widehat{K}_{|\mathcal{S}|}^{2}} \|\nabla F(\xx^{(k)})\|^{2}.
\end{equation*} 
Now the result follows, using Assumption~\eqref{F_strong_gen}, as in the end of the proof of Theorem~\ref{global_newton}.
\qed
\end{proof}

\begin{proof}[Proof of Theorem~\ref{global_newton_spectral}:]
By the choice of $\lambda^{(k)}$ and the convexity of $f_{i}$, we have $\lambda^{(k)} > 0$ and, so by~\eqref{angle_inexact} it gives 
\begin{equation*}
\pp_{k}^{T} \nabla F(\xx^{(k)}) \leq -(1-\theta_{2}) \pp_{k}^{T} \hat{H}(\xx^{(k)}) \pp_{k} \leq - (1-\theta_{2})\lambda^{(k)} \|\pp_{k}\|^{2}.
\end{equation*}
So it follows that $\pp_{k}^{T} \nabla F(\xx^{(k)}) < 0$ and we can indeed obtain decrease in the objective function.
Now as before, in order to pass the Armijo rule, we search for $\alpha$ such that
\begin{equation*}
\alpha \frac{K}{2} \|\pp_{k}\|^{2} \leq -(1-\beta) \pp_{k}^{T}\nabla F(\xx^{(k)}),
\end{equation*}
which, in turn, is satisfied if 
\begin{equation*}
\alpha   \leq \frac{2(1-\theta_{2}) (1-\beta)\lambda^{(k)}}{K}.
\end{equation*} 
Now the rest of the proof is similar to that of Theorem~\ref{global_newton_inexact} by noting that
\begin{equation*}
[\hat{H}(\xx^{(k)})]^{-1} \preceq \frac{1}{\lambda^{(k)}},
\end{equation*}
and
\begin{equation*}
\nabla F(\xx^{(k)})^{T} [\hat{H}(\xx^{(k)})]^{-1} \nabla F(\xx^{(k)}) \geq \frac{1}{\max\{\widehat{K}_{|\mathcal{S}|},\lambda^{(k)}\}} \|\nabla F(\xx^{(k)})\|^{2}.
\end{equation*}
\qed
\end{proof}

\begin{proof}[Proof of Theorem~\ref{global_newton_ridge}:]
As before, by the choice of $\lambda > 0$, convexity of $f_{i}$, and~\eqref{angle_inexact} we get
\begin{equation*}
\pp_{k}^{T} \nabla F(\xx^{(k)}) \leq -(1-\theta_{2}) \pp_{k}^{T} \hat{H}(\xx^{(k)}) \pp_{k} \leq - (1-\theta_{2})\lambda \|\pp_{k}\|^{2}.
\end{equation*}
%\begin{equation*}
%\pp_{k}^{T} \nabla F(\xx^{(k)}) = -\pp_{k}^{T} \hat{H}(\xx^{(k)}) \pp_{k} \leq -\lambda\|\pp_{k}\|^{2},
%\end{equation*}
which implies that $\pp_{k}^{T} \bgg(\xx^{(k)}) < 0$ and we can indeed obtain decrease in the objective function. Similarly to the proof previous theorems, it is easy to see that 
\begin{equation*}
\alpha   \leq \frac{2(1-\theta_{2})(1-\beta)\lambda}{K},
\end{equation*}
satisfies the Armijo rule. The rest of the results also follow as in the proof of Theorem~\ref{global_newton_inexact} by noting that
\begin{equation*}
[\hat{H}(\xx^{(k)})]^{-1} \preceq \frac{1}{\lambda},
\end{equation*}
and
\begin{equation*}
\nabla F(\xx^{(k)})^{T} [\hat{H}(\xx^{(k)})]^{-1} \nabla F(\xx^{(k)}) \geq \frac{1}{\widehat{K}_{|\mathcal{S}|} + \lambda} \|\nabla F(\xx^{(k)})\|^{2}.
\end{equation*}
\qed
\end{proof}

\subsection{Proofs of Section~\ref{sec:subsampl_newton_hess_grad}}
The proof of the following lemma, in a more general format for constrained optimization, is given in the companion paper, SSN2~\cite[Lemma 4]{romassn2}. However, it is given here as well for completeness.
\begin{proof}[Proof of Lemma~\ref{randnla_lemma}]
As mentioned before, the full gradient, $\nabla F(\xx)$, can be equivalently written as a product of two matrices as $\nabla F(\xx) = A B$, where
\begin{eqnarray*}
A &\defeq& \begin{pmatrix}
\mid & \mid & & \mid \\
\nabla f_{1}(\xx) & \nabla f_{2}(\xx) & \cdots & \nabla f_{n}(\xx)\\
\mid & \mid & & \mid \\
\end{pmatrix} \in \mathbb{R}^{p \times n}, \\
B &\defeq& \left(1/n,1/n,\ldots,1/n\right)^{T} \in \mathbb{R}^{n \times 1}.
\end{eqnarray*}
As a result, approximating the gradient using sub-sampling is equivalent to approximating the product $AB$ by sampling columns and rows of A and B, respectively, and forming matrices $\widehat{A}$ and $\widehat{B}$ such $\widehat{A}\widehat{B} \approx AB$. More precisely, for a random sampling index set $\mathcal{S}$, we can represent the sub-sampled gradient~\eqref{subsampled_G}, by the product $\widehat{A} \widehat{B}$ where $\widehat{A} \in \mathbb{R}^{p \times |\mathcal{S}|}$ and $\widehat{B} \in \mathbb{R}^{|\mathcal{S}| \times 1}$ are formed by selecting uniformly at random and with replacement, $|\mathcal{S}|$ columns and rows of $A$ and $B$, respectively, rescaled by $\sqrt{n/|\mathcal{S}|}$. Now, by the assumption on $G(\xx)$, we can use~\cite[Lemma 11]{drineas2006fast} to get
\begin{equation*}
\|AB - \widehat{A}\widehat{B} \|_{F} = \|\nabla F(\xx) - \bgg(\xx) \| \leq \frac{G(\xx)}{\sqrt{|\mathcal{S}|}} \Big( 1 + \sqrt{8 \ln \frac{1}{\delta}}\Big),
\end{equation*}
with probability $1-\delta$. Now the result follows by requiring that 
\begin{equation*}
\frac{G(\xx)}{\sqrt{|\mathcal{S}|}} \Big( 1 + \sqrt{8 \ln \frac{1}{\delta}}\Big) \leq \epsilon.
\end{equation*}
\qed
\end{proof}

\begin{proof}[Proof of Theorem~\ref{global_newton_grad}:]
We give the proof only for the case of sampling without replacement. The proof for sampling with replacement is obtained similarly.

As in the proof of Theorem~\ref{global_newton}, we first need to show that there exists an iteration-independent step-size, $\widetilde{\alpha} > 0$, such that the constrain in~\eqref{global_alpha_grad} holds for any $0 \leq \alpha \leq \widetilde{\alpha}$. For any $0 \leq \alpha < 1$, define $\xx_{\alpha} = \xx^{(k)} + \alpha \pp_{k}$. By Assumption~\eqref{F_strong_gen}, we have

\begin{eqnarray*}
F(\xx_{\alpha}) - F(\xx^{(k)}) &\leq& (\xx_{\alpha} - \xx^{(k)})^{T}\nabla F(\xx^{(k)}) + \frac{K}{2} \|\xx_{\alpha} - \xx^{(k)}\|^{2} \\
&=& \alpha \pp_{k}^{T}\nabla F(\xx^{(k)}) + \alpha^{2} \frac{K}{2} \|\pp_{k}\|^{2} \\
&=& \alpha \pp_{k}^{T} \bgg(\xx^{(k)}) + \alpha \pp_{k}^{T}(\nabla F(\xx^{(k)}) - \bgg(\xx^{(k)})) + \alpha^{2} \frac{K}{2} \|\pp_{k}\|^{2} \\
&\leq& \alpha \pp_{k}^{T} \bgg(\xx^{(k)}) + \alpha\|\nabla F(\xx^{(k)}) - \bgg(\xx^{(k)})\| \|\pp_{k}\| + \alpha^{2} \frac{K}{2} \|\pp_{k}\|^{2} \\
&\leq& \alpha \pp_{k}^{T} \bgg(\xx^{(k)}) + \epsilon_{2} \alpha\|\pp_{k}\| + \alpha^{2} \frac{K}{2} \|\pp_{k}\|^{2}.
\end{eqnarray*}
By~\eqref{eig_chernoff} and~\eqref{global_p}, we have 
\begin{equation*}
\pp_{k}^{T} \bgg(\xx^{(k)}) = - \pp_{k}^{T} H(\xx^{(k)}) \pp_{k} \geq -(1-\epsilon_{1}) \gamma \|\pp_{k}\|^{2},
\end{equation*}
which shows that $\pp_{k}^{T} \bgg(\xx^{(k)}) < 0$ and we can indeed obtain decrease in the objective function.
Now, using the above, it follows that
\begin{eqnarray*}
F(\xx_{\alpha}) - F(\xx^{(k)}) \leq -\alpha \pp_{k}^{T} H(\xx^{(k)}) \pp_{k} + \alpha \epsilon_{2} \|\pp_{k}\|  + \alpha^{2} \frac{K}{2} \|\pp_{k}\|^{2}.
\end{eqnarray*}
As a result, we need to search for $\alpha$ such that
\begin{equation*}
-\alpha \pp_{k}^{T} H(\xx^{(k)}) \pp_{k} + \epsilon_{2} \alpha \|\pp_{k}\|  + \alpha^{2} \frac{K}{2} \|\pp_{k}\|^{2} \leq -\alpha \beta\pp_{k}^{T} H(\xx^{(k)}) \pp_{k},
\end{equation*}
which follows if 
\begin{equation*}
\epsilon_{2}   + \alpha \frac{K}{2} \|\pp_{k}\| \leq (1 -\beta )(1-\epsilon_{1}) \gamma \|\pp_{k}\|.
\end{equation*}
This latter inequality holds if 
\begin{eqnarray*}
\alpha &=& \frac{(1- \beta)(1-\epsilon_{1})\gamma}{K}, \\
\epsilon_{2}   &=& \frac{(1 -\beta )(1-\epsilon_{1}) \gamma}{2} \|\pp_{k}\|.
\end{eqnarray*}
Hence, from $H(\xx^{(k)}) \pp_{k} = -\bgg(\xx^{(k)})$, it follows that in order to guarantee an iteration independent lower bound for $\alpha$ as above, we need to have
\begin{equation*}
\epsilon_{2}   \leq \frac{(1 -\beta )(1-\epsilon_{1}) \gamma \|\bgg(\xx^{(k)})\| }{2  \widehat{K}_{|\mathcal{S}|} },
\end{equation*}
which, by the choice of $\sigma$ and $\epsilon_{1}$, is imposed by the algorithm. If the stopping criterion succeeds, then by
\begin{equation*}
\|\bgg(\xx^{(k)})\| \geq \|\nabla F(\xx^{(k)})\| - \epsilon_{2} 
\end{equation*}
it follows that, 
\begin{equation*}
\|\nabla F(\xx^{(k)})\| < \left(1 + \sigma \right)\epsilon_{2}.
\end{equation*}
However, if the stopping criterion fails and the algorithm is allowed to continue, then by
\begin{equation*}
\|\bgg(\xx^{(k)})\| \leq \|\nabla F(\xx^{(k)})\| + \epsilon_{2}, 
\end{equation*}
it follows that 
\begin{eqnarray*}
\left(\sigma - 1\right) \epsilon_{2}   \leq \|\nabla F(\xx^{(k)})\|.
\end{eqnarray*}
Now, since $\sigma \geq 4$, we get that 
\begin{eqnarray*}
\frac{2}{3} \|\nabla F(\xx^{(k)})\| \leq   \left(\frac{\sigma-2}{\sigma - 1 }\right) \|\nabla F(\xx^{(k)})\|  \leq \|\nabla F(\xx^{(k)})\| - \epsilon_{2}.
\end{eqnarray*}
Hence, from $H(\xx^{(k)}) \pp_{k} = -\bgg(\xx^{(k)})$, we get
\begin{eqnarray*}
\pp_{k}^{T} H(\xx^{(k)}) \pp_{k} &=& \bgg(\xx^{(k)})^{T} [H(\xx^{(k)})]^{-1} \bgg(\xx^{(k)}) \\
&\geq& \frac{1}{ \widehat{K}_{|\mathcal{S}|}} \|\bgg(\xx^{(k)})\|^{2} \\
&\geq& \frac{1}{ \widehat{K}_{|\mathcal{S}|}} \left(\|\nabla F(\xx^{(k)})\| - \epsilon_{2}\right)^{2} \\
%&\geq& \left(\frac{\sigma-2}{\sigma - 1 }\right)^{2} \frac{1}{ K} \|\nabla F(\xx^{(k)})\|^{2} \\
&\geq& \frac{4}{ 9\widehat{K}_{|\mathcal{S}|}} \|\nabla F(\xx^{(k)})\|^{2}.
\end{eqnarray*}
Finally, using Assumption~\eqref{F_strong_gen}, the desired result follows as in the end of the proof of Theorem~\ref{global_newton}.\qed
\end{proof}

\begin{proof}[Proof of Theorem~\ref{local_global_newton_grad}]
The choice of $\epsilon_{1}$ and $\epsilon_{2}^{(k)}$ is to meet a requirement of SSN2~\cite[Theorem 13]{romassn2} and account for the differences between Lemma~\ref{chernoff_lemma} and SSN2~\cite[Lemma 1]{romassn2}. 

As in the proof of Theorem~\ref{local_global_newton}, we get
\begin{equation*}
\widehat{F}(\alpha) \leq \widehat{F}(0) + \alpha \widehat{F}^{'}(0) + \frac{\alpha^{2}}{2} \widehat{F}^{''}(0) + \frac{\alpha^{3}}{6} L \|\pp_{k}\|^{3}.
\end{equation*}
On the other hand, we have
\begin{align*}
\|\pp_{k}\|^{2} = \|[H(\xx^{(k)})]^{-1} \bgg(\xx^{(k)})\|^{2} \leq \frac{1}{(1-\epsilon_{1})\gamma} \bgg(\xx^{(k)})^{T} [H(\xx^{(k)})]^{-1} \bgg(\xx^{(k)}).
\end{align*}
In addition, from $\|\nabla F(\xx^{(k)}) - \bgg(\xx^{(k)})\| \leq \epsilon_{2}$, we get $\pp_{k}^{T} \nabla F(\xx^{(k)}) \leq \pp_{k}^{T} \bgg(\xx^{(k)})\| + \epsilon_{2} \|\pp_{k}\|$ and so 
\begin{align*}
\widehat{F}^{'}(0) = \alpha \pp_{k}^{T} \nabla F(\xx^{(k)}) &\leq \alpha \pp_{k}^{T} \bgg(\xx^{(k)})\| + \alpha\epsilon_{2} \|\pp_{k}\| \\
& = -\alpha \bgg(\xx^{(k)})^{T} [H(\xx^{(k)})]^{-1} \bgg(\xx^{(k)})\| + \alpha\epsilon_{2} \|\pp_{k}\|
\end{align*}
 Finally, as in the proof of Theorem~\ref{local_global_newton}, we have
\begin{align*}
\widehat{F}^{''}(0) = \alpha^{2} \pp_{k}^{T} \nabla^{2} F(\xx^{(k)}) \pp_{k}  \leq \frac{\alpha^{2}}{(1-\epsilon_{1})} \bgg(\xx^{(k)})^{T} [H(\xx^{(k)})]^{-1} \bgg(\xx^{(k)}).
\end{align*}
Hence, with $\alpha = 1$ and denoting $h(\xx) \defeq \bgg(\xx)^{T} [H(\xx)]^{-1} \bgg(\xx)$, we have
\begin{align*}
& F(\xx_{\alpha}) \leq F(\xx^{(k)}) + \left(\frac{1}{2 (1-\epsilon_{1})} -1 \right) h(\xx^{(k)}) + \frac{L}{6}\left( \frac{1}{(1-\epsilon_{1})\gamma} h(\xx^{(k)})\right)^{3/2} + \epsilon_{2} \left( \frac{1}{(1-\epsilon_{1})\gamma} h(\xx^{(k)})\right)^{1/2} \\
& \leq  F(\xx^{(k)}) + h(\xx^{(k)}) \left(\frac{1}{2 (1-\epsilon_{1})} -1 + \frac{L}{6} \left(\frac{1}{(1-\epsilon_{1})\gamma}\right)^{3/2} h(\xx^{(k)})^{1/2} + \epsilon_{2}\left(\frac{1}{(1-\epsilon_{1})\gamma}\right)^{1/2} h(\xx^{(k)})^{-1/2} \right)\\
& \leq F(\xx^{(k)}) + h(\xx^{(k)}) \left(\frac{1}{2 (1-\epsilon_{1})} -1 + \frac{L}{6 (1-\epsilon_{1})^{2}\gamma^{2}} \|\bgg(\xx^{(k)})\| + \epsilon_{2}\left(\frac{\tilde{\kappa}}{(1-\epsilon_{1})}\right)^{1/2} \|\bgg(\xx^{(k)})\|^{-1} \right),
\end{align*}
where the last inequality follows from $\|\bgg(\xx)\|^{2}/\widehat{K}_{|\mathcal{S}|} \leq h(\xx) \leq \|\bgg(\xx)\|^{2}/((1-\epsilon_{1})\gamma)$. 
Now denoting
\begin{align*}
A &\defeq \frac{L}{6 (1-\epsilon_{1})^{2}\gamma^{2}} \\
B &\defeq \frac{1}{2 (1-\epsilon_{1})} -1 + \beta \\
C &\defeq \epsilon_{2}\left(\frac{\tilde{\kappa}}{(1-\epsilon_{1})}\right)^{1/2}\\
y &\defeq \|\bgg(\xx^{(k)})\|,
\end{align*}
we require that
\begin{equation*}
A y^{2} + B y + C \leq 0.
\end{equation*}
The roots of this polynomial are
\begin{align*}
y & = \frac{-B \pm \sqrt{B^{2} - 4AC}}{2 A }\\
& = \frac{(1 - \beta - \frac{1}{2 (1-\epsilon_{1})}) \pm \sqrt{(1 - \beta - \frac{1}{2 (1-\epsilon_{1})})^{2} - \frac{4 L \epsilon_{2}}{6 (1-\epsilon_{1})^{2}\gamma^{2}} \left(\frac{\tilde{\kappa}}{(1-\epsilon_{1})}\right)^{1/2}}}{\frac{2 L}{6 (1-\epsilon_{1})^{2}\gamma^{2}}} \\
& = \frac{(1 -2\epsilon_{1} - 2 (1-\epsilon_{1}) \beta) \pm \sqrt{(1 -2\epsilon_{1} - 2 (1-\epsilon_{1}) \beta)^{2} -  \frac{8L \epsilon_{2}}{3 \gamma^{2}} \left(\frac{\tilde{\kappa}}{(1-\epsilon_{1})}\right)^{1/2}}}{\frac{2L}{3 (1-\epsilon_{1})\gamma^{2}}} \\
& = \frac{3 (1-\epsilon_{1})\gamma^{2} (1 -2\epsilon_{1} - 2 (1-\epsilon_{1}) \beta) \pm \sqrt{ 9 (1-\epsilon_{1})^{2}\gamma^{4} (1 -2\epsilon_{1} - 2 (1-\epsilon_{1}) \beta)^{2} -  9 (1-\epsilon_{1})^{3/2} \gamma^{4} \frac{8L \epsilon_{2} \sqrt{\tilde{\kappa}} }{3 \gamma^{2}} }}{2L} \\
& = \frac{3 (1-\epsilon_{1})\gamma^{2} (1 -2\epsilon_{1} - 2 (1-\epsilon_{1}) \beta) \pm \sqrt{ 9 (1-\epsilon_{1})^{2}\gamma^{4} (1 -2\epsilon_{1} - 2 (1-\epsilon_{1}) \beta)^{2} -  24 (1-\epsilon_{1})^{3/2} \gamma^{2} L \epsilon_{2} \tilde{\kappa}^{1/2}}}{2L}.
\end{align*}
Define
\begin{subequations}
\label{q1_q2}
\begin{align}
q_{1}(\epsilon_{1},\epsilon_{2},\beta,\tilde{\kappa},L)  &\defeq \frac{q - \sqrt{ q^{2} -  24 (1-\epsilon_{1})^{3/2} \gamma^{2} L \epsilon_{2} \tilde{\kappa}^{1/2}}}{2L} \label{q1},\\
q_{2}(\epsilon_{1},\epsilon_{2},\beta,\tilde{\kappa},L)  &\defeq \frac{q + \sqrt{ q^{2} -  24 (1-\epsilon_{1})^{3/2} \gamma^{2} L \epsilon_{2} \tilde{\kappa}^{1/2}}}{2L} \label{q2},
\end{align}
\end{subequations}
where $q \defeq 3 (1-\epsilon_{1})\gamma^{2} (1 -2\epsilon_{1} - 2 (1-\epsilon_{1}) \beta)$.
It is easy to see that $q_{1}(\epsilon_{1},\epsilon_{2},\beta,\tilde{\kappa},L)$ is increasing with $\epsilon_{2}$ with $q_{1}(\epsilon_{1},0,\beta,\tilde{\kappa},L) = 0$, while $q_{2}(\epsilon_{1},\epsilon_{2},\beta,\tilde{\kappa},L)$ is decreasing with $\epsilon_{2}$ with $q_{2}(\epsilon_{1},0,\beta,\tilde{\kappa},L)$ being equal to the right hand side of~\eqref{grad_small}. In order to ensure that $q_{1}$ and $q_{2}$ are real, we also need to have
\begin{equation*}
\epsilon_{2}  \leq \frac{3 \sqrt{(1-\epsilon_{1})}\gamma^{2} (1 -2\epsilon_{1} - 2 (1-\epsilon_{1}) \beta)^{2}}{8  L \sqrt{\tilde{\kappa}}}.
\end{equation*}

Now if
\begin{equation}
\label{grad_small_2}
q_{1}(\epsilon_{1},\epsilon_{2},\beta,\tilde{\kappa},L) \leq \|\bgg(\xx^{(k)})\| \leq  q_{2}(\epsilon_{1},\epsilon_{2},\beta,\tilde{\kappa},L),
\end{equation}
we get
\begin{align*}
F(\xx^{(k)} + \pp_{k}) \leq F(\xx^{(k)}) -\beta \bgg(\xx^{(k)})^{T} [H(\xx^{(k)})]^{-1} \bgg(\xx^{(k)}) = F(\xx^{(k)}) +\beta \pp_{k}^{T} \bgg(\xx^{(k)}),
\end{align*}
which implies that~\eqref{global_alpha_grad} is satisfied with $\alpha = 1$. Note that the left hand side of~\eqref{grad_small_2} is enforced by the stopping criterion of the algorithm as for any $\epsilon_{2}$, $q_{1}(\epsilon_{1},\epsilon_{2},\beta,\tilde{\kappa},L)\leq \sigma \epsilon_{2}$. The proof is complete if we can find $k$ such that both the sufficient condition of SSN2~\cite[Theorem 13]{romassn2} as well as the right hand side of~\eqref{grad_small_2} is satisfied. First note that from Theorem~\ref{global_newton_grad}, Assumption~\eqref{F_strong_gen} and by using the iteration-independent lower bound on $\alpha_{k}$, it follows that
\begin{equation*}
\|\nabla^{2}F(\xx^{(k)})\|^{2} \leq 2 K (1 - \hat{\rho})^{k} \left(F(\xx^{(0)}) - F(\xx^{*}) \right),
\end{equation*}
where
\begin{equation*}
\hat{\rho} = \frac{8 \beta  (1- \beta)(1-\epsilon_{1}) }{9 \tilde{\kappa} \kappa}.
\end{equation*}
Now, if the stopping criterion fails and the algorithm is allowed to continue, then by
\begin{equation*}
\|\bgg(\xx^{(k)})\| \leq \|\nabla F(\xx^{(k)})\| + \epsilon_{2}, 
\end{equation*}
we get
\begin{eqnarray*}
\left(\sigma - 1\right) \epsilon_{2}   \leq \|\nabla F(\xx^{(k)})\|,
\end{eqnarray*}
which implies that
\begin{equation*}
\|\bgg(\xx^{(k)})\| \leq \frac{\sigma}{\sigma-1} \|\nabla F(\xx^{(k)})\| \leq 2 \|\nabla F(\xx^{(k)})\|. 
\end{equation*}
As a result, in order to satisfy the right hand side of~\eqref{grad_small_2},  we require that 
\begin{equation*}
8 K (1 - \hat{\rho})^{k} \left(F(\xx^{(0)}) - F(\xx^{*}) \right) \leq \frac{16}{9}( \rho_{2} - (\rho_{0} + \rho_{1}) )^{2} q_{2}^{2}(\epsilon_{1},\epsilon_{2},\beta,\tilde{\kappa},L),
\end{equation*}
which yields~\eqref{local_k_grad}. Again, from Theorem~\ref{global_newton_grad} and Assumtpion~\eqref{F_strong_gen},we get
\begin{equation*}
\|\xx^{(k)} - \xx^{*}\|^{2} \leq \frac{2 (1 - \hat{\rho})^{k}}{\gamma} \left(F(\xx^{(0)}) - F(\xx^{*}) \right),
\end{equation*}
which implies that
\begin{align*}
\|\xx^{(k)} - \xx^{*}\|^{2} &\leq \frac{ 16( \rho_{2} - (\rho_{0} + \rho_{1}) )^{2} q_{2}^{2}(\epsilon_{1},\epsilon_{2},\beta,\tilde{\kappa},L),}{36 \gamma K } \\
&\leq \frac{4 ( \rho_{2} - (\rho_{0} + \rho_{1}) )^{2} (1-\epsilon_{1})^{2}\gamma^{4} (1 -2\epsilon_{1} - 2 (1-\epsilon_{1}) \beta)^{2} }{ \gamma K L^{2}} \\
& \leq \frac{4 ( \rho_{2} - (\rho_{0} + \rho_{1}) )^{2} (1-\epsilon_{1})^{2}\gamma^{2}  }{L^{2}} = c^{2},
\end{align*}
and hence the sufficient condition of SSN2~\cite[Theorem 13]{romassn2} is also satisfied and we get~\eqref{local_rate_grad}. \qed
\end{proof}

\begin{proof}[Proof of Theorem~\ref{global_newton_grad_inexact}:]
The proof is given by combining the arguments used to prove Theorems~\ref{global_newton_inexact} and~\ref{global_newton_grad}, and is given here only for completeness. We also give the proof only for the case of sampling without replacement.
The proof for sampling with replacement is obtained similarly.

As in the proof of Theorem~\ref{global_newton_grad}, we get
\begin{eqnarray*}
F(\xx_{\alpha}) - F(\xx^{(k)}) \leq \alpha \pp_{k}^{T} \bgg(\xx^{(k)}) + \epsilon_{2} \alpha\|\pp_{k}\| + \alpha^{2} \frac{K}{2} \|\pp_{k}\|^{2},
\end{eqnarray*}
and
\begin{eqnarray}
\pp_{k}^{T} \bgg(\xx^{(k)}) &\leq& -(1-\theta_{2})(1-\epsilon_{1}) \gamma \|\pp_{k}\|^{2}.
\label{armijo_loose}
\end{eqnarray}
Hence, $\pp_{k}^{T} \bgg(\xx^{(k)}) < 0$ and we can indeed obtain decrease in the objective function.
For the Armijo rule to hold, we search for $\alpha$ such that
\begin{equation*}
\epsilon_{2} \|\pp_{k}\| + \alpha \frac{K}{2} \|\pp_{k}\|^{2} \leq -(1-\beta) \pp_{k}^{T}\nabla F(\xx^{(k)}).
\end{equation*}
which follows if 
\begin{equation*}
\epsilon_{2} + \alpha \frac{K}{2} \|\pp_{k}\| \leq (1-\theta_{2})(1-\beta) (1-\epsilon_{1}) \gamma \|\pp_{k}\|,
\end{equation*}
which, in turn, is satisfied by having 
\begin{eqnarray*}
\alpha &=& \frac{(1-\theta_{2})(1- \beta)(1-\epsilon_{1})\gamma}{K}, \\
\epsilon_{2}   &=& \frac{(1-\theta_{2})(1 -\beta )(1-\epsilon_{1}) \gamma}{2} \|\pp_{k}\|.
\end{eqnarray*}
Now $\|H(\xx^{(k)})\pp_{k} + \bgg(\xx^{(k)})\| \leq \theta_{1} \|\bgg(\xx^{(k)})\|$ implies
\begin{eqnarray*}
\|\pp_{k}\| \geq \frac{(1 -\theta_{1}) \|\bgg(\xx^{(k)})\|}{\widehat{K}_{|\mathcal{S}|}},
\end{eqnarray*}
and hence, we need to have
\begin{equation*}
\epsilon_{2}   \leq \frac{(1 -\theta_{1}) (1 -\theta_{2})  (1 -\beta )(1-\epsilon_{1}) \gamma \|\bgg(\xx^{(k)})\| }{2  \widehat{K}_{|\mathcal{S}|} },
\end{equation*}
which, by the choice of $\sigma$ and $\epsilon_{1}$, is imposed by the algorithm. If the stopping criterion holds, then by 
\begin{equation*}
\|\bgg(\xx^{(k)})\| \geq \|\nabla F(\xx^{(k)})\| - \epsilon_{2},
\end{equation*}
it follows that 
\begin{equation*}
\|\nabla F(\xx^{(k)})\| < \left(1 + \sigma \right)\epsilon_{2}.
\end{equation*}
However, if the stopping criterion fails and the algorithm continues, then by
\begin{equation*}
\|\bgg(\xx^{(k)})\| \leq \|\nabla F(\xx^{(k)})\| + \epsilon_{2}, 
\end{equation*}
it follows that 
\begin{eqnarray*}
%&&\epsilon_{2}   \leq \frac{(1 -\beta )(1-\epsilon_{1}) \gamma \|\bgg(\xx^{(k)})\| }{2  K } \\
%&&\epsilon_{2}   \leq \frac{(1 -\beta )(1-\epsilon_{1}) \gamma (\|\nabla F(\xx^{(k)})\| + \epsilon_{2}) }{2  K } \\
(\sigma - 1) \epsilon_{2}   \leq \|\nabla F(\xx^{(k)})\|,
\end{eqnarray*}
which, since $\sigma \geq 4$, implies that 
\begin{eqnarray*}
\frac{2}{3}\|\nabla F(\xx^{(k)})\|  \leq   \left(\frac{\sigma-2}{\sigma - 1 }\right) \|\nabla F(\xx^{(k)})\|  \leq \|\nabla F(\xx^{(k)})\| - \epsilon_{2}.
\end{eqnarray*}
For part (i), we notice that by the deﬁnition of the vector $\ell_{2}$ norm, i.e.,
\begin{equation*}
\|\vv\|_{2} = \sup \{\ww^{T} \vv; \; \|\ww\|_{2}=1\},
\end{equation*}
it follows that the condition~\eqref{global_p_inexact}, implies
\begin{equation*}
\pp_{k}^{T} \bgg(\xx^{(k)}) + \bgg(\xx^{(k)})^{T} [H(\xx^{(k)})]^{-1} \bgg(\xx^{(k)}) \leq \theta_{1} \|\bgg(\xx^{(k)})\| \|[H(\xx^{(k)})]^{-1}\bgg(\xx^{(k)})\|.
\end{equation*}
Now as in the proof of Theorem~\ref{global_newton_inexact}, we get that if
\begin{equation*}
\theta_{1} \leq \frac{\sqrt{(1-\epsilon_{1})\gamma}}{2\sqrt{\widehat{K}_{|\mathcal{S}|}}},
\end{equation*}
then
\begin{equation*}
\pp_{k}^{T} \bgg(\xx^{(k)}) \leq \frac{-1}{2 \widehat{K}_{|\mathcal{S}|}} \|\bgg(\xx^{(k)})\|^{2}.
\end{equation*}
Since $\|\nabla F(\xx^{(k)})\| - \epsilon_{2}  \leq \|\bgg(\xx^{(k)})\|$, we get
\begin{equation*}
\pp_{k}^{T} \bgg(\xx^{(k)}) \leq  - \frac{1}{2 \widehat{K}_{|\mathcal{S}|}} \left(\|\nabla F(\xx^{(k)})\| - \epsilon_{2} \right)^{2} \leq - \frac{2}{9 \widehat{K}_{|\mathcal{S}|}}  \|\nabla F(\xx^{(k)})\|^{2}.
\end{equation*}

For part (ii), we note that by $\|H(\xx^{(k)})\pp_{k} + \bgg(\xx^{(k)})\| \leq \theta_{1} \|\bgg(\xx^{(k)})\|$, we get
\begin{eqnarray*}
\|\pp_{k}\| \geq \frac{(1 -\theta_{1})}{\widehat{K}_{|\mathcal{S}|}} \left(\|\nabla F(\xx^{(k)})\| - \epsilon_{2}\right) \geq \frac{2(1 -\theta_{1})}{3\widehat{K}_{|\mathcal{S}|}} \|\nabla F(\xx^{(k)})\|.
\end{eqnarray*}
we then square both sides and use~\eqref{armijo_loose} to get
\begin{eqnarray*}
\pp_{k}^{T} \bgg(\xx^{(k)}) &\leq& -\frac{4(1 -\theta_{1})^{2}(1-\theta_{2})(1-\epsilon_{1}) \gamma}{9\widehat{K}_{|\mathcal{S}|}^{2}} \|\nabla F(\xx^{(k)})\|^{2}
\end{eqnarray*}

Now the result follows, using Assumption~\eqref{F_strong_gen}, as in the end of the proof of Theorem~\ref{global_newton}.
\qed
\end{proof}

\end{document}